\newcounter{NiveauDetails}
\documentclass[11pt,dvipsnames, english]{amsart} 
\usepackage[active]{srcltx}
\usepackage[dvipsnames,usenames]{xcolor}
\usepackage[utf8]{inputenc} 
\usepackage[T1]{fontenc}
\usepackage{amsfonts}
\usepackage{euscript,amsthm,multirow,mathrsfs,lmodern,pstricks,enumerate}
\usepackage[lite,alphabetic]{amsrefs}
\ifnum \arabic{NiveauDetails} > 0 \usepackage{showlabels} \else\fi

\theoremstyle{plain}
\newtheorem{theo}{Theorem}[section]
\newtheorem{lemm}[theo]{Lemma}
\newtheorem{prop}[theo]{Proposition}
\newtheorem{coro}[theo]{Corollary}

\newtheorem*{theoA}{Theorem A}

\newtheorem*{coroB}{Corollary B}
\newtheorem*{coroC}{Corollary C}

\newtheorem*{theorem*}{Theorem }

\theoremstyle{definition}
\newtheorem{defi}[theo]{Definition}

\newtheorem{nota}[theo]{Notation}
\newtheorem{exem}[theo]{Example}
\newtheorem{assu}[theo]{Assumption}

\newtheorem{rema}[theo]{Remark}

\def\cal{\mathcal}

\def\a{\alpha}
\def\b{\beta}

\def\g{\gamma}
\def\G{\Gamma}
\def\rhoc{\rho_{\rm cpt}}

\def\ad{{\rm ad}}

\def\deg{{\rm deg}\,}

\def\SU{{\rm SU}}
\def\Spin{{\rm Spin}}
\def\GL{{\rm GL}}
\def\End{{\rm End}}
\def\Hom{{\rm Hom}}
\def\U{{\rm U}}
\def\SL{{\rm SL}}

\def\rk{{\rm rk}\hspace{.5pt}}
\def\Ker{{\rm Ker\,}}
\def\Im{{\rm Im\,}}
\def\vol{{\rm vol}}
\def\Sp{{\rm Sp}}

\def\pfd{\rightarrow}

\def\t{\theta}
\def\ds{\displaystyle}

\def\o{\omega}

\def\L{{\mathcal L}}
\def\F{{\mathcal F}}
\def\T{{\mathcal T}}
\def\S{{\mathcal S}}

\def\gg{{\mathfrak g}}
\def\kg{{\mathfrak k}}

\def\ug{{\mathfrak u}}

\def\r{\rho}

\def\V{{\mathbb V}}
\def\W{{\mathbb W}}
\def\E{{\mathbb E}}

\def\H{{{\mathbb H}^n_{\mathbb C}}}

\def\PTX{{{\mathbb P}T_X}}
\def\PTH{{\mathbb P}T_{{\mathbb H}^n_{\mathbb C}}}

\def\OX{{\mathcal O}_{\PTX}(-1)}

\def\F{{\mathcal F}}
\def\CP{{\mathbb{CP}}}

\def\su{{\mathfrak{su}}}

\newcommand{\suma}[1]{\alpha_1^\vee + \cdots + \alpha_{#1}^\vee}
\newcommand{\av}{\alpha^\vee}

\newcommand{\Z}{\mathbb Z}
\newcommand{\R}{\mathbb{R}}
\newcommand{\Q}{\mathbb{Q}}
\newcommand{\C}{\mathbb{C}}

\newcommand{\pe}[1]{\ifnum \arabic{NiveauDetails} > 0 \color{OliveGreen} #1 \color{black} \else\fi}

\newcommand{\indication}[1]{\ifnum \arabic{NiveauDetails} > 1 \color{Maroon} #1 \color{black} \else\fi}

\newcommand{\commentaire}[1]{}

\newcommand{\lpara}{\vskip 2mm}
\newcommand{\mpara}{\vskip 5mm}

\def\qed{\hfill $\square$}

 \newcommand{\fe}{\mathfrak e} 
\newcommand{\fg}{\mathfrak g} \newcommand{\fh}{\mathfrak h} 
 \newcommand{\fk}{\mathfrak k} \newcommand{\fl}{\mathfrak l}
  
\newcommand{\fp}{\mathfrak p} \newcommand{\fq}{\mathfrak q} 
 \newcommand{\ft}{\mathfrak t} \newcommand{\fu}{\mathfrak u}
  
 \newcommand{\fz}{\mathfrak z}
\newcommand{\fsl}{\mathfrak{sl}}

\def\bA{{\mathbb A}} \def\bB{{\mathbb B}} \def\bC{{\mathbb C}} 
\def\bE{{\mathbb E}} \def\bF{{\mathbb F}}  
   \def\bO{{\mathbb O}}

   \def\bV{{\mathbb V}}

\newcommand{\poidsesixnu}[6]{
\hspace{-.12cm}
\begin{array}{ccccc}
{} \hspace{0cm} #1 & {} \hspace{-.1cm} #2 & {} \hspace{-.1cm} #3 &
{} \hspace{-.1cm} #4 & {} \hspace{-.1cm} #5 \vspace{-.13cm}\\
\hspace{-.0cm} & \hspace{-.1cm} & {} \hspace{-.1cm} #6 &
{} \hspace{-.1cm} & {} \hspace{-.1cm}
\end{array}
\hspace{-.2cm}     }

\newcommand{\poidseseptnu}[7]{
\hspace{-.12cm}
\begin{array}{cccccc}
{} \hspace{0cm} #1 & {} \hspace{-.1cm} #2 & {} \hspace{-.1cm} #3 &
{} \hspace{-.1cm} #4 & {} \hspace{-.1cm} #5
& {} \hspace{-.1cm} #6 \vspace{-.13cm}\\
\hspace{-.0cm} & \hspace{-.1cm} & {} \hspace{-.1cm} #7 &
{} \hspace{-.1cm} & {} \hspace{-.1cm} & {} \hspace{-.1cm}
\end{array}
\hspace{-.2cm}     }

\newcommand{\p}{{\mathbb P}}
\renewcommand{\P}{{\mathbb P}}

\renewcommand{\iff}{if and only if }

\newcommand{\scal}[1]{\langle #1 \rangle}
\newcommand{\im}{\mathrm{Im}\,}

\renewcommand{\det}{\mathrm{det}}

\newcommand{\codim}{\mbox{codim}}

\def\cC{\mathcal{C}}  \def\cE{\mathcal{E}}
   
 \def\cL{\mathcal{L}}  
\def\cO{\mathcal{O}}

\def\cV{\mathcal{V}} \def\cX{\mathcal{X}}  

\def\bA{{\mathbb A}} \def\bB{{\mathbb B}} \def\bC{{\mathbb C}} 
\def\bE{{\mathbb E}} \def\bF{{\mathbb F}}  
   \def\bO{{\mathbb O}}

   \def\bU{{\mathbb U}} \def\bV{{\mathbb V}}

 \def\fh{\mathfrak{h}}
 \def\ft{\mathfrak{t}} \def\fu{\mathfrak{u}}
\def\fz{\mathfrak{z}}

\newcommand{\zmax}{z_{\mathtt{max}}}
\newcommand{\Zmax}{Z_{\mathtt{max}}}
\newcommand{\Ek}{\E^{\otimes k}}
\newcommand{\Vk}{\V^{\otimes k}}
\newcommand{\cEk}{\cE^{\otimes k}}

\def\ovE{\bar{E}}
\def\ovt{\bar{\theta}}
\def\ovb{\bar{\beta}}
\def\ovg{\bar{\gamma}}

\begin{document}

\title[Maximal representations in exceptional Hermitian groups]{Maximal representations\\ of uniform complex hyperbolic lattices\\ in exceptional Hermitian Lie groups}
\author{Pierre-Emmanuel Chaput, Julien Maubon}
\maketitle

\tableofcontents

\section{Introduction}

This paper deals with maximal representations of complex hyperbolic lattices in semisimple Hermitian
Lie groups with no compact factors.

A {\em complex hyperbolic lattice} $\G$ is a lattice in the Lie group $\SU(1,n)$, a finite
cover of the group of biholomorphisms of the $n$-dimensional complex hyperbolic space
$\H=\SU(1,n)/\U(n)$. Unless otherwise
specified, we shall always
assume that our lattice $\G$ is {\em uniform}, meaning that the quotient $X:=\G\backslash\H$ is
compact, and that it  is {\em torsion free}, so that $X$ is also a
manifold. The Kähler form on $X$ induced by the
$\SU(1,n)$-invariant Kähler form on $\H$ with constant holomorphic sectional curvature $-1$ will be
denoted by $\o$.

A semisimple Lie group with no compact factors $G_\R$ is said to be {\em Hermitian} if its
associated symmetric space $M=G_\R/K_\R$ is a Hermitian symmetric space
(of the noncompact type). This means that the symmetric space $M$ admits a $G_\R$-invariant complex
structure, which makes it a Kähler manifold.
We will call $\o_M$ the $G_\R$-invariant Kähler form of
$M$, normalized so that its holomorphic sectional curvatures lie between $-1$ and $-\frac 1{\rk M}$,
where $\rk M$ is the rank of the symmetric space $M$, or equivalently the real rank $\rk_\R G_\R$ of
$G_\R$. We will also assume that the complexification $G$ of $G_\R$ is simply connected.

If $\rho$ is a representation (a group homomorphism) from $\G$ to $G_\R$, we can
define its {\em Toledo
invariant} $\tau(\rho)$ as follows:
$$
\tau(\rho)=\frac 1{n!}\int_X f^\star\o_M\wedge\o^{n-1}, 
$$
where $f:\H\pfd M$ is a $\cC^\infty$ and $\rho$-equivariant map and $f^\star\o_M$ is seen as a 2-form on $X$ by
$\G$-invariance. The Toledo invariant does not depend on the choice of the map $f$, it depends only
on $\rho$, and in fact, only on the connected component of $\rho$ in $\Hom(\G,G_\R)$. Moreover, it
satifies the following Milnor-Wood type inequality:
$$
|\tau(\rho)|\leq \rk M\,\vol(X),
$$
a fundamental property established in full generality in~\cite{BI07}.

{\em Maximal representations} $\rho:\G\pfd G_\R$ are those representations
for which the Milnor-Wood inequality is an equality.

\mpara

In \cite{KM}, Koziarz and the second named author classified maximal representations when $G_\R$ is a
classical group.
In the present work, we extend this classification to all Hermitian groups, and we prove:

\begin{theoA}
  Let $\G$ be a torsion free uniform lattice in $\SU(1,n)$, $n\geq 2$, and let $\rho$ be a maximal representation
  of $\G$ in a Hermitian Lie group $G_\R$.

  Then there exists a unique $\rho$-equivariant harmonic map $f$ from
  $\H$ to the associated symmetric space $M$. If $\tau(\rho)>0$, $f$  is holomorphic and satisfies
  $f^\star\o_M=\rk M\,\o$. If $\tau(\rho)<0$, $f$ is antiholomorphic and satisfies $f^\star\o_M=-\rk M\,\o$.
  Moreover $f$ is a totally geodesic embedding.
\end{theoA}

It follows quite easily from this that the map $f$ of the theorem is \emph{tight}.
Such maps between Hermitian symmetric spaces were classified in~\cite{Hamlet}, and
from his classification we deduce:

\begin{coroB}
  Under the assumptions of Theorem A, each simple factor of $G_\R$ is either isogenous to
  $\SU(p,q)$ for some $(p,q)$ with $q\geq n\, p$,
  or to the exceptional group ${\rm E_{6(-14)}}$, the latter being possible only if $n=2$.
\end{coroB}

 We can also deduce a structure result
for the representation $\rho$. The map $f$ of Theorem~$A$ is in fact equivariant w.r.t. a
  morphism of Lie groups $\varphi:\SU(1,n) \to G_\R$.  We denote by $H_\R$ the image of $\varphi$ in $G_\R$
and  by  $Z_\R$ the centralizer of $H_\R$ in $G_\R$.

\begin{coroC}
Under the assumptions of Theorem A,  the representation $\rho$ is reductive, discrete, faithful, and acts cocompactly on the image of
  $f$ in $M$. The centralizer $Z_\R$ is compact and
  there exists a group morphism $\rhoc : \G \to Z_\R$ such that
  $$ \forall \gamma \in \Gamma, \,\rho(\gamma) = \varphi(\gamma)\rhoc(\gamma) = \rhoc(\gamma)\varphi(\gamma).$$
\end{coroC}

\noindent
Moreover, Lemma \ref{lemm:fixator} says that $Z_\R$ is exactly the subgroup of $G_\R$ fixing all the
points of $f(\H)$. Its isomorphism class is described in Lemma \ref{lemm:Z}.



\begin{rema}
  The assumption that $\G$ is torsion free is technical and can be removed by passing to a normal
  finite index subgroup of $\G$. Indeed, using Selberg Lemma, one can always find a torsion free
  normal finite index subgroup $\G'$ of $\G$. Theorem A and Corollaries B and C are therefore
  applicable to $\G'$. We have chosen to assume that the complexification $G$
  of $G_\R$ is simply-connected to simplify the exposition, but this assumption is not necessary:
  see Remark \ref{rema:1-connected}.  
\end{rema}

\medskip

The global strategy we
adopt here is the same as in~\cite{KM}: we consider a Higgs bundle $(\bar E,\bar\t)$ on the quotient
$X=\G\backslash\H$ associated
to a (reductive) representation $\rho:\G\pfd G_\R$ (see~\ref{subsec:Higgs}) and we translate the
Milnor-Wood inequality into an inequality involving degrees of subbundles of $\bar E$
(see~\ref{subsec:reword}). This inequality is then proved (in~\ref{subsec:proofMW}) using the
Higgs-stability properties of
$(\bar E,\bar\t)$,
or rather the {\em leafwise} Higgs-stability properties of the pull-back $(E,\t)$ of
$(\bar E,\bar\t)$ to the
projectivized tangent bundle $\PTX$ of $X$ with respect to the {\em tautological foliation} on $\PTX$
(see Subsections~\ref{subsec:HiggsFol} and~\ref{subsec:toto}).

Although classical target groups were already treated in \cite{KM}, we decided not to focus
immediately on the exceptional cases and instead to provide a
more unified perspective, as independent as possible of the
classification of the simple Hermitian
Lie groups, in the spirit of \cite{biquard}. To achieve this, instead of considering the Higgs
vector bundle associated to the standard representation of the complexification $G$ of $G_\R$
(which is only defined in the classical cases), we
work with  the Higgs
vector bundle $(E,\t)$ associated to the {\em cominuscule} representation
$\E$ of $G$ such that the dual compact symmetric space $M^\vee=G/P$ of $M$ is embedded in
the projectivization $\mathbb P\E$ of $\E$: this is sometimes called the {\em first canonical
embedding} of $M^\vee$, see \cite{nakagawa}*{p. 651}.

On the algebraic side, we present in
Section~\ref{sec:subnil} a general
construction of a graded submodule of $\E$ associated to an element of $\fu_-$ if
$\fg = \fk \oplus (\fu_+ \oplus \fu_-)$ is the Cartan decomposition of the Lie algebra of
$G$.
On the geometric side, this construction gives leafwise 
  Higgs subsheaves of $(E,\t)\pfd\PTX$ associated to the components
of the Higgs field $\t$ (see Subsection~\ref{subsec:HiggsSubsheaf}), whose existence is then used to
prove the 
Milnor-Wood inequality. To be a bit more precise, on a {\em generic} fiber of $(E,\t)\pfd\PTX$,  the
leafwise Higgs subsheaves we produce admit
  purely representation theoretic descriptions. The algebraic counterparts of the generic objects
  are first introduced and studied in 
 Section~\ref{sec:subnil}. This is then used in~\ref{subsec:HiggsSubsheaf} to define the subsheaves
 and prove that they have the desired properties.

This unified approach allows to exclude the possibility of maximal representations in {\em tube type}
target Lie groups, and in particular in ${\rm E}_{7 (-25)}$. Indeed in this case
the representation $\rho$ satisfies an inequality stronger than the Milnor-Wood inequality
(Proposition~\ref{prop:tube}). Maximal representations in ${\rm E}_{6 (-14)}$ are treated
in Subsection~\ref{subsec:E6} where we prove that they can exist only if $n=2$, in which case they are
essentially induced by a homomorphism $\SU(2,1)\pfd {\rm E}_{6 (-14)}$, see Theorem~\ref{thm:E6}.

\mpara

In \cite{BIW}, the authors introduced the notion of \emph{tight} representations.
By \cite{BI07}*{Lemma 5.3}, any maximal representation is tight, and by
\cite{BIW}*{Theorem 3}, a tight representation is {\em reductive}, meaning
that the Zariski closure of its image in $G_\R$
is a reductive subgroup of $G_\R$. This reduces the study of maximal representations
to the reductive ones.
Furthermore, by e.g. \cite{KM}*{Lemma 4.11}, any representation can be deformed to a reductive one with
the same Toledo invariant.
This also reduces the proof of the Milnor-Wood inequality to the case of reductive representations.
From now on, we therefore assume without loss of generality that
\begin{assu}\label{assu:reductive}
 The representation $\rho:\G\pfd G_\R$ is reductive.
\end{assu}

\section{Submodule of a cominuscule representation associated with a nilpotent
  element}\label{sec:subnil}

Here we develop the algebraic material that we will need in Section~\ref{sec:MW} to give a new and
unified proof of the Milnor-Wood inequality.

\medskip

Let $G_\R$ be a simple noncompact Hermitian Lie group and $K_\R$ a maximal compact subgroup of
$G_\R$. The associated irreducible Hermitian symmetric space $G_\R/K_\R$ will be denoted by the
letter $M$.

Let also
$G = G_\R \otimes \C$ and $K = K_\R \otimes \C$ be the complexifications of the real algebraic
groups $G_\R$ and $K_\R$. We assume that $G$ is simply connected.
Let $T \subset K$ be a maximal torus of $G$.  We denote by $\fg,\fk,\ft$ the corresponding complex Lie
algebras. Let $R$ be the set of roots of $G$, $\Pi$ be a basis of $R$, and $W$ the Weyl group of $R$. The center
$\fz\subset\ft$ of $\fk$ is $1$-dimensional and we let  $z 
\neq 0$ be an element in this center.  We have the Cartan  decomposition
$\fg = \fk \oplus \fu$ 
, where

$$
\fk = \ft \oplus \bigoplus_{\alpha : \scal{\alpha,z}=0} \fg_\alpha \ \ , \ \
\fu = \bigoplus_{\alpha : \scal{\alpha,z}\neq 0} \fg_\alpha.
$$

Now the adjoint action of $z$ on $\fu$ gives the complex structure of the Hermitian symmetric space $M=G_\R / K_\R$: $\ad(z)_{|\fu}$ has therefore exactly two opposite eigenvalues and, up to scaling $z$, we assume that these are $\pm 2$. The corresponding eigenspaces  $\fu_+$ and $\fu_-$ are Abelian, their root space decompositions are
$$
 \fu_+ = \bigoplus_{\alpha : \scal{\alpha,z}=2} \fg_\alpha \ \ , \ \
 \fu_- = \bigoplus_{\alpha : \scal{\alpha,z}=-2} \fg_\alpha \ .
$$


\begin{nota}
 The \emph{rank} of the symmetric space $M$, or equivalently the real rank of $G_\R$, will be
 denoted by $p$.
A root $\a$ of $\fg$ is {\em noncompact} if $\scal{\a,z}\neq 0$.   Linearly independent positive
noncompact roots $\a_1,\ldots,\a_r$ are said to be {\em strongly orthogonal} if for all $i\neq j$,
$\a_i\pm\a_j$ is not a root.
 All maximal sets of strongly orthogonal roots of $\fu_+$ have cardinality $p$ (\cite{Harish}, or
 \cite{Helgason}*{Ch. VIII, \textsection 7}). 

For a root $\a\in R$, we let $\a^\vee$ be the associated coroot and $\varpi_\a$ the
  fundamental weight corresponding to $\a$. We denote by $R^\vee=\{\a^\vee\mid \a\in R\}$ the dual
  root system. The set of simple roots $\Pi$ is a basis of $\ft^\star$,
  $\Pi^\vee=\{\b^\vee\mid\b\in\Pi\}$ is a basis of $\ft$ and $\{\varpi_\a\mid \a\in\Pi\}$ is the
  dual basis of $\Pi^\vee$. If $\a\in R$, we write $\a=\sum_{\b\in\Pi}n_\b(\a)\b$ the expression of
  $\a$ in terms of the simple roots. 

 In this paper, we use the convention that if the root system $R$ of $\fg$ is simply laced, then all
 the roots are long. Therefore short roots exist only if $R$ is not simply laced. We recall that 
  there can be at most 
 two root lengths in $R$ and that if there are two root lengths the ratio 
 $\frac{\text{long 
 root length}}{\text{short root length}}$ equals $\sqrt{2}$  (Hermitian symmetric spaces exist
only in type $A_n$, $B_n$, $C_n$, $D_n$, $E_6$ or $E_7$). 
 
\end{nota}

\subsection{The cominuscule representation and its grading}\label{subsec:minuscule-rep}\hfill
\medskip

There exists a unique simple noncompact root $\zeta\in\Pi$. It follows from the classification that
$\zeta$ is long, see Table~\ref{equa:tableau}. The root $\zeta$, or equivalently $z$, defines the
parabolic subalgebra $$\ds \fp=\fk\oplus\fu^+=\ft\oplus\bigoplus_{\a:n_\zeta(\a)\geq 0}\fg_\a$$ of
$\fg$ and hence a 
parabolic subgroup $P$ of $G$. The projective variety $M^\vee=G/P$ is a Hermitian symmetric space of
compact type called the compact dual of $M=G_\R/K_\R$. 

Let $\varpi:=\varpi_\zeta$ be the fundamental weight associated to the noncompact simple root $\zeta$ and
let $\E$ be the irreducible representation of $G$ whose highest weight is $\varpi$. Let $\E_\varpi$
be the $\varpi$-eigenspace of $\E$. Then $\E_\varpi$ is 1-dimensional and one can check that its
stabilizer in $G$ is $P$. This gives a $G$-equivariant holomorphic and isometric embedding of
$M^\vee=G/P$ in the projective space $\P\E$. It is called the {\em first canonical embedding} of
$M^\vee$. See e.g.~\cite{nakagawa} for more details.

By \cite{Murakami}, since $G$ is simply connected, the Picard group ${\rm Pic}(G/P)$ is isomorphic to
the group of characters $\cX(P)$ of $P$. Since $\fp=\fz\oplus[\fp,\fp]$, $\cX(P)$ is isomorphic to
$\Z$ and thus it is generated by the smallest positive character of $P$, namely $\varpi$. 
Moreover, the isomorphism $\iota:\cX(P) \simeq {\rm Pic}(G/P)$ is given by $\lambda \mapsto (G
\times \C_\lambda)/P$, 
where $\C_\lambda$ denotes the $1$-dimensional $P$-module defined by $\lambda$.
Since $\E_\varpi\simeq \C_{\varpi}$ as $P$-modules,
we see that $\cL:=\iota(\varpi) \simeq \cO_{\P \E}(1)_{|M^\vee}$ is a generator of ${\rm Pic}(M^\vee)$.

\medskip

\begin{defi}
A fundamental weight $\varpi_\b$, $\b\in \Pi$, is {\em minuscule} w.r.t. the root system $R$ if
$\scal{\varpi_\b,\a^\vee}\in\{-1,0,1\}$ for all $\a\in R$ (cf. e.g. \cite{bourbaki}*{Chapter VI,
  Exercise 24}). It is {\em cominuscule} if the fundamental coweight $\varpi_{\b^\vee}$ associated
to the coroot $\b^\vee\in\Pi^\vee$ is minuscule w.r.t. the dual root system $R^\vee$. 

An irreducible representation of $G$ whose highest weight is a (co)minuscule fundamental weight of
$R$ is also called (co)minuscule. 
\end{defi}

\begin{rema}
If $R$ is simply laced, then it is isomorphic to $R^\vee$ and hence the cominuscule property is
equivalent to the minuscule property. 
\end{rema}

Since the fundamental coweights $\{\varpi_{\b^\vee}\mid \b\in\Pi\}$ form by definition the dual basis
of the coroots of $\Pi^\vee$, i.e. of $\Pi$, a fundamental weight $\varpi_\b$ is cominuscule if and
only if $n_\b(\a)\in\{-1,0,1\}$ for all $\a\in R$. 
   
We conclude immediately that the weight $\varpi:=\varpi_\zeta$, and hence the $G$-representation $\E$, are
cominuscule. 
Let indeed $\a$ be a root.
As we saw, $\scal{\a,z}\in\{-2,0,2\}$ and $\zeta$ is the only simple noncompact root, so that
$\scal{\a,z}=n_\zeta(\a)\scal{\zeta,z}=2n_\zeta(\a)$. Hence the
coefficient $n_\zeta(\a)$  belongs to $\{-1,0,1\}$. Equivalently, we can observe that the coweight
$\varpi_{\zeta^\vee}$ is $\frac 12z$, which gives the result.


\lpara

We now begin our study of the cominuscule representation $\E$ of $G$.

\begin{nota}
 We denote by $\mu_0$ the lowest weight of $\E$ and by $X(\E)$ the set
 of weights of $\E$. For $\chi \in X(\E)$, we write
 $\E_\chi$ for the corresponding weight space. Recall that $\varpi$ is the highest weight of $\E$.
\end{nota}

The fact that $\E$ is cominuscule has the following consequence on the weights of $\E$:

\begin{lemm}
 \label{lemm:weight2}
 For any weight $\chi$ of $\E$, and any root $\a\in R$,  $|\scal{\chi,\alpha^\vee}| \leq 2$, and the equality
 $|\scal{\chi,\alpha^\vee}| = 2$ implies that $\alpha$ is short.
\end{lemm}

\begin{proof}
  For the highest weight $\varpi$ of $\E$, the results follows from the fact that
  $\scal{\varpi,\alpha^\vee}=n_\zeta(\a)\frac{\|\zeta\|^2}{\|\a\|^2}$. The result still holds if
  $\varpi$ is replaced by $w \cdot \varpi$, where $w \in W$ is arbitrary, and since any weight of
  $\E$ is in the convex hull of $W \cdot \varpi$, it holds for any weight. 
\end{proof}

We deduce that the structure of $\E$ with respect to the action of $\fg_\alpha$ for $\alpha$ a long root is particularly simple:

\begin{lemm}
 \label{lemm:gV}
 Let $\alpha$ be a long root and let $\chi$ be a weight of $\E$.
 We have
 $$
 \fg_{-\alpha} \cdot \E_\chi = \left \{
  \begin{array}{ll}
   \E_{\chi-\alpha} \mbox{ if } \scal{\chi,\alpha^\vee} = 1\, , \\
   \{ 0 \} \mbox{ otherwise}\,.
  \end{array}
 \right .
 $$
\end{lemm}
\begin{proof}
Let $\alpha$ be long and let $\fsl_2(\alpha)$ be the Lie subalgebra of $\fg$ corresponding to $\alpha$.
Let $M \subset \E$ be the $\fsl_2(\alpha)$-submodule generated by $\E_\chi$. By Lemma \ref{lemm:weight2},
any irreducible component $V$ of $M$ is a $\fsl_2(\alpha)$-module of dimension $1$ or $2$.

We therefore have
only three possibilities. The first case is when $V=V_\chi$, $\fg_\alpha \cdot V_\chi = \{0\}$ and
$\fg_{-\alpha} \cdot V_\chi = \{0\}$. In this case, $\scal{\chi,\alpha^\vee}=0$.
The second case is when $V=V_\chi \oplus V_{\chi - \alpha}$, $\fg_\alpha \cdot V_\chi = \{0\}$ and
$\fg_{-\alpha} \cdot V_\chi = V_{\chi-\alpha}$. In this case, $\scal{\chi,\alpha^\vee}=1$ (and
$\scal{\chi-\alpha,\alpha^\vee}=-1$).
The third (symmetric) case is when $V=V_\chi \oplus V_{\chi + \alpha}$, $\fg_\alpha \cdot V_\chi = V_{\chi + \alpha}$ and
$\fg_{-\alpha} \cdot V_\chi = \{0\}$. In this case, $\scal{\chi,\alpha^\vee}=-1$ (and
$\scal{\chi+\alpha,\alpha^\vee}=1$).

If $\scal{\chi,\alpha^\vee} = 1$, we deduce that $M=M_\chi \oplus M_{\chi-\alpha}$ and that
$\fg_{-\alpha} \cdot M_\chi = M_{\chi-\alpha}$. We have $M_\chi = \E_\chi$ so $\dim(\E_\chi) \leq \dim(\E_{\chi-\alpha})$.
Arguing with the $\fsl_2(\alpha)$-submodule $M'$
generated by $\E_{\chi-\alpha}$, we see that the dimensions are equal which implies
$M_{\chi-\alpha}=\E_{\chi-\alpha}$. The lemma is proved in this case.
If $\scal{\chi,\alpha^\vee} \leq 0$, we see that $\fg_{-\alpha} \cdot \E_\chi = \{0\}$.
%
\end{proof}

\begin{nota}
  We denote by $\zmax$ the integer $\scal{\varpi,z}$. One can observe that $\zmax=2\,\frac{\dim
    M^\vee}{c_1(M^\vee)}$ (see, e.g., \cite{KMgeneraltype}*{Section 2}).
\end{nota}

\begin{prop}
 \label{prop:rang}
 The set $\{\scal{\chi,z} \ | \ \chi \in X(\E) \}$ is the set $\{\zmax,\zmax-2,\ldots,\zmax-2p\}$.
\end{prop}
\begin{proof}
 It follows from \cite{richardson}*{Theorem 2.1} that the $W$-orbit of the weight $\varpi$ is
 exactly the set of weights of the form
 $\varpi - \sum_{i=1}^k \a_i$, where $(\a_i)_{1 \leq i \leq k}$ is a family of long strongly orthogonal roots.
 For any $i$, we have $\scal{\a_i,z}=2$, thus we have the equality of sets
 $$\{\scal{\mu,z} \, | \, \mu \in W \cdot \varpi \} = \{ \zmax, \zmax-2 , \ldots, \zmax-2p  \}\ .$$
 In particular, $\scal{\mu_0,z} = \zmax-2p$ and for $\chi \in X(\E)$, we have
 $\zmax-2p \leq \scal{\chi,z} \leq \zmax$. The result of the proposition now follows from the fact that
 $2$ is a divisor of $\scal{\a,z}$ for any root $\a$.
\end{proof}

\lpara

Now we can introduce the grading of $\E$:

\begin{defi}
 \label{defi:Ei}
 For a relative integer $i$, let $\displaystyle \E_i := \bigoplus_{\chi : \scal{\chi,z} = \zmax-2i} \E_\chi $.
\end{defi}

This grading corresponds to the decomposition of $\E$ into irreducible $K$-modules:

\begin{prop}
 \label{prop:Ei-irr}
 The $K$-modules $\E_i$ are irreducible.
\end{prop}

\begin{proof}
This might be well-known to experts, but we include a proof for completeness.
We give a case by case argument.
In all types except in type $C_n$, the representation $\E$ is well enough understood, so that we
get readily the result. In fact, in type $A_{n-1}$, we have $\E = \wedge^p (\C^p \oplus \C^{n-p})$
and thus $\E_i = \wedge^{p-i} \C^p \otimes \wedge^i \C^{n-p}$:  this is an irreducible ${\rm S}(\GL_p
\times \GL_{n-p})$-module. In type $B_n$, we have $\E = \C^{2n+1} = \C \oplus \C^{2n-1} \oplus \C$,
and each summand is an irreducible ${\rm Spin}_{2n-1}$-module, hence an irreducible $K$-module.
In type $(D_n,\varpi_1)$ the situation is similar.

In type $(D_n,\varpi_n)$, $\E$ is the spinor 
representation of the spin group and, according to \cite{chevalley}, we have $\E  = \wedge^0 \C^n
\oplus \wedge^2 \C^n \oplus \cdots \oplus \wedge^{2p} \C^n$ (where $p=[n/2]$). Thus $\E_i =
\wedge^{2i} \C^n$, and this is an irreducible $\GL_n$-module. For the types $E_6$ and $E_7$, we use
models of these exceptional Lie algebras and their minuscule representations, as given for example 
in \cite{manivel}. In type $E_6$, we have $\E = \C \oplus V_{16} \oplus V_{10}$, where $V_{16}$ is
a spinor representation and $V_{10}$ the vector representation of the spin group ${\rm Spin_{10}}$. In
type $E_7$, we have $\E = \C \oplus V_{27} \oplus V'_{27} \oplus \C$, where $V_{27}$ and $V'_{27}$
are the two minuscule representations of a group of type $E_6$. In both cases, the $K$-modules
$\E_i$ are irreducible. 

We now deal with the case of type $C_n$.
We denote by $\C^{2n} = \C^n_a \oplus \C^n_b$ a symplectic $2n$-dimensional space, with $\C^n_a$ and
$\C^n_b$ supplementary isotropic subspaces. We then have $\E=\left ( \wedge^n \C^{2n} \right
)_\omega$, where the symbol $\omega$ means that we take in $\wedge^n \C^{2n}$ the irreducible
${\rm Sp}_{2n}$-submodule containing the highest weight line $\wedge^n \C^n_a$.

 We first claim that for any $i$, the variety $M^\vee \cap \p \E_i$ generates $\p \E_i$ as a
 projective space.
 To prove this claim, we set $\bF_i \subset \E_i$ to be the space generated by the affine cone over
 $M^\vee \cap \E_i$. We denote by $\b_1,\ldots,\b_n$ the base of the root system.
 We consider $\bF = \bigoplus_i \bF_i$. Then $\bF$ is obviously $K$-stable. Moreover, if $x \in
 M^\vee \cap \p \E_i$, then applying Lemma \ref{lemm:weight2}, since $\b_n$ is long, the
 $\SL_2(\b_n)$-orbit of $x$ is $\{x\}$ itself or a line $\scal{x,y}$ joining $x$ and a point $y$ in
 $\E_{i-1}$ or $\E_{i+1}$. In either case, $y$ belongs to $M^\vee$, and this implies that $\bF$ is
 $\SL_2(\b_n)$-stable. 

 This implies that $\bF$ is $\Sp_{2n}$-stable, so $\bF = \E$, and $\bF_i=\E_i$. To prove that
 $\E_i=\bF_i$ is an irreducible 
 $K$-module, it is enough to show that $M^\vee \cap \p \E_i$ is a single $K$-orbit.
 Note that the previous part of the argument would be valid
 for any cominuscule representation, but we now give a specific argument in the case of $\Sp_{2n}$ to
 show that 
 $M^\vee \cap \p \E_i$ is a single $K$-orbit. In fact, $\E_i$ is a submodule of
 $\wedge^{n-i} \C^n_a \otimes \wedge^i \C^n_b$, and $M^\vee \cap \p \E_i$ represents the set of
 Lagrangian subspaces of 
 $\C^{2n}$ which meet $\C^n_a$ in dimension $n-i$ and $\C^n_b$ in dimension $i$. Since
 $\C^n_b$ is the dual space to $\C^n_a$, via the symplectic form $\omega$, such a space $\Lambda$ is
 equal to the direct sum of $\Lambda \cap \C^n_a$ and $(\Lambda \cap \C^n_a)^\bot$ (with $(\Lambda \cap 
 \C^n_a)^\bot \subset \C^n_b$). 
 We deduce that
 $M^\vee \cap \p \E_i$ is isomorphic to the Grassmannian of $(n-i)$-subspaces in $\C^n_a$, and thus
 it is a single $\GL_n$-orbit.
\end{proof}

\begin{prop}
 \label{prop:Ei}
 We have the following properties:
 \begin{enumerate}[(a)]
  \item $\E = \E_0 \oplus \E_{1} \oplus \cdots \oplus \E_{p}$.
  \item $\E_0 = \E_{\varpi}$.
  \item $\E_{i+1} = \fu_- \cdot \E_i$.
  \item The map $\E_0 \otimes \fu_- \to \E_1$ is an isomorphism.
 \end{enumerate}
\end{prop}
\begin{proof}
 Only the last two points need a proof. Let $U(\fu_-)$ denote the envelopping algebra of $\fu_-$. The third point
 follows from the fact that $\E = U(\fu_-) \cdot \E_{\varpi}$ and the fact that for $\alpha$ a root of $\fu_-$,
 we have $\scal{\alpha,z} = -2$. 
 The last point follows by Schur's lemma since $\fu_-$ and $\E_1$ are irreducible $\fk$-modules.
\end{proof}

\subsection{Rank of a nilpotent element, dominant orthogonal sequences}
\label{subsec:DOS}\hfill\medskip

We now consider an element $y \in \fu_-$ and we describe a particularly nice representative of its
orbit under $K$.

\medskip

The $K$-orbits in $\fu_-$
are parametrized by integers $r \in \{0,\ldots,p\}$ and
a representative of each orbit is $y_{\alpha_1} + \cdots + y_{\alpha_r}$, where the roots $\alpha_i$
are strongly orthogonal and long, and $y_{\alpha_i}$ is a fixed element in the root space of $\a_i$ (see
e.g. \cite{Harish}, \cite{Helgason}*{Ch. VIII, \textsection 7}, or \cite{Wolf}). 
If $y\in\fu_-$ is in the orbit corresponding to the integer $r$,
$r$ is called the {\em rank} of $y$.


\begin{rema}
 In case $\fg$ has type $A$, $\fu_-$ identifies with a space of matrices, and
 the rank as defined above of an element in $\fu_-$ coincides with its rank as a matrix.
\end{rema}

The following proposition is a characterization of the rank $r$ of $y\in\fu_-$ in terms
of its action on $\E$:

\begin{prop}
 \label{prop:rang-y}
 Let $y = y_{\alpha_1} + \cdots + y_{\alpha_r}$ with $\a_1,\ldots,\a_r$ a family of strongly orthogonal long roots,
 and let $\varphi(y):\E \to \E$ be the corresponding linear map.
 We have $\varphi(y)^{r+1}=0$ and $\varphi(y)^r(\E_{0}) = \E_{\varpi-\alpha_1-\cdots-\alpha_r}$, so
 in particular $\varphi(y)^r(\E_{0}) \neq \{ 0 \}$.
\end{prop}

\begin{proof}
 We have $y_{\alpha_i} \in \fg_{-\alpha_i}$. For any pair $(\chi,\alpha)$ with $\chi \in X(\E)$ and
 $\a$ a long root,
 $\chi-2\alpha$ cannot be a weight of $\E$ by Lemma \ref{lemm:weight2} and its proof. We deduce that
 $\varphi(y_{\alpha_i})^2=0$. On the other hand, for any $i,j$ with $i \neq j$,
 the maps $\varphi(y_{\alpha_i})$ and $\varphi(y_{\alpha_j})$ commute since $\alpha_i + \alpha_j$ is not
 a root. Thus we get
 $$
 \varphi(y)^k = k!\ \sum \varphi(y_{\alpha_{j_1}}) \circ \cdots \circ \varphi(y_{\alpha_{j_k}}) \ ,
 $$
 where the sum is over the increasing sequences $1\leq j_1 < \cdots < j_k \leq r$.

 In particular $\varphi(y)^{r+1} = 0$, and
 $\varphi(y)^r = r!\ \varphi(y_{\alpha_1}) \circ \cdots \circ \varphi(y_{\alpha_r})$.
 By Lemma \ref{lemm:gV}, we have
 $\varphi(y_{\alpha_1}) \circ \cdots \circ \varphi(y_{\alpha_r}) \cdot \E_{\varpi} = \E_{\varpi-\alpha_1-\cdots-\alpha_r}$,
 so the proposition is proved.
\end{proof}


\medskip

In \cite{kostant}, Kostant introduced his so-called ``chain cascade'' of orthogonal roots.
Here we will need a version of his algorithm where we impose that all the roots of the chain cascade
have a positive coefficient on $\zeta$. Note that a similar algorithm is used in \cite{buch}.

\begin{defi}\label{def:cascade}
We define an integer $q$ and, for any integer $i$
such that $1 \leq i \leq q$,
a root $\alpha_i$ together with a subset $\Pi_i \subset \Pi$, by the following inductive process:
\begin{itemize}
 \item We let $\Pi_1=\Pi$.
 \item Assuming that $\alpha_1,\ldots,\alpha_{i-1}$ and $\Pi_1,\ldots,\Pi_i$ have been defined,
 we let $\alpha_i$ be the highest root of the root system $R(\Pi_i)$ generated by $\Pi_i$.
 \item We let $\Sigma_i \subset \Pi_i$ be the set of simple roots $\beta$ such that
 \begin{equation}
  \label{equa:sigma_i}
  \scal{\alpha_i^\vee,\beta} \neq 0\, .
 \end{equation}
 \item If $\zeta \in \Sigma_i$, then $q=i$ and the algorithm terminates. Otherwise, $\Pi_{i+1}$
 is the connected component of $\Pi_i \setminus \Sigma_i$ containing $\zeta$. 
\end{itemize}
If $(\alpha_i)_{1 \leq i \leq q}$ is the sequence defined by this process,
we say that it is the {\em maximal dominant orthogonal sequence} for $\varpi$. More generally,
  the sequences $(\alpha_i)_{1 \leq i \leq r}$ for $1\leq
  r\leq q$ are called the {\em dominant orthogonal sequences}.
\end{defi}

The following proposition essentially adapts the results of \cite{kostant} to our context and
explains our terminology.

\begin{prop}
 \label{prop:cascade}
 Let $(\alpha_i)_{1 \leq i \leq q}$ be the maximal dominant orthogonal sequence for $\varpi$. Then
 \begin{enumerate}
  \item The roots $\alpha_i$ are long and strongly orthogonal.
  \item $q=p$.
  \item For any integer $i \leq p$, $\alpha_1^\vee + \cdots + \alpha_i^\vee$ is a dominant coweight.
  \item $\varpi - \alpha_1 - \cdots - \alpha_p$ is the lowest weight $\mu_0$ of the irreducible
    $G$-module $\E$.
 \end{enumerate}
\end{prop}
\begin{proof}
The root $\alpha_i$ is the highest root of $R(\Pi_i)$, and $\Pi_i$ contains the long root $\zeta$, so $\alpha_i$
is long. By construction, $\scal{\alpha_i^\vee,\alpha_j}=0$ if $j>i$, so $\alpha_i$ and $\alpha_j$ are
orthogonal. Since they are both long, they are strongly orthogonal. Kostant \cite[Lemma 1.6]{kostant} also proved the strong
orthogonality. This proves (1).

For (2), let us prove that $(\alpha_1,\ldots,\alpha_q)$ is a maximal sequence of orthogonal roots
(see also \cite[Theorem 1.8]{kostant}). Let
$\alpha \in R$ be such that $\scal{\alpha_i^\vee,\alpha}=0$ holds for all $i$, and let us assume that
$\alpha >0$. Let $i$ be the greatest integer such that $\alpha \in R(\Pi_i)$. By maximality of $i$ there exists
a simple root $\beta$ in $Supp(\alpha) \cap \Sigma_i$. Since $\alpha_i$ is dominant on $R(\Pi_i)$, we have
$\scal{\alpha_i^\vee,\alpha} \geq \scal{\alpha_i^\vee,\beta} > 0$, a contradiction to the existence of
$\alpha$.

For the third point, let $i \leq p$ and let $\beta \in \Pi$.
\begin{itemize}
 \item If $\beta \in \Pi_i$, by construction, $\scal{\alpha_j^\vee,\beta}=0$ for $j<i$. Since $\alpha_i$ is the
 highest root of $\Pi_i$, $\scal{\alpha_i^\vee,\beta} \geq 0$. Thus
 $\scal{\suma{i},\beta} = \scal{\av_i,\beta} \geq 0$.
 \item If $\beta \in \Sigma_{i-1}$, then $\scal{\av_{i-1},\beta} \geq 1$ and $\scal{\av_j,\beta}=0$ for $j<i+1$.
 Since $\alpha_i$ is long, $\scal{\av_i,\beta} \geq -1$. Thus, $\scal{\suma{i},\beta} \geq 0$.
 \item If $\beta \in \Sigma_j$ for $j<i-1$, then, by construction of $\Pi_i$,
 $\scal{\av_i,\beta}=0$. By induction on $i$,
 $\scal{\suma{i-1},\beta} \geq 0$, so $\scal{\suma{i},\beta} \geq 0$.
\end{itemize}
 
For (4), we observe that $s_{\alpha_i}(\varpi)=\varpi-\alpha_i$ since $\alpha_i$ is long, by Lemma \ref{lemm:weight2}.
Thus $s_{\alpha_1} \cdots s_{\alpha_p}(\varpi)=\varpi-\alpha_1 - \cdots - \alpha_p$ is a weight of $\E$.
We prove that it is a lowest weight. Let $\beta \in \Pi$. If $\beta \neq \zeta$, then
$\scal{\varpi-\alpha_1 - \cdots - \alpha_p,\beta^\vee} = \scal{-\alpha_1 - \cdots - \alpha_p,\beta^\vee}$, and
this is non-positive by (3) since all the roots $\alpha_i$ have the same length. If $\beta=\zeta$, then
we compute that $\scal{\varpi-\alpha_1 - \cdots - \alpha_p,\zeta^\vee} = 1 - \scal{\alpha_p,\zeta^\vee} \leq 0$
since $\zeta \in \Sigma_p$.
\end{proof}

For the convenience of the reader and later use, we recall in Table~\ref{equa:tableau} below
what we obtain applying this
recursive construction in all cases. We don't indicate
all the roots $\alpha_i$, but rather the sum of the corresponding coroots
$\alpha_1^\vee + \cdots + \alpha_r^\vee$, by indicating in the shape of
a Dynkin diagram the values
$\scal{\alpha_1^\vee + \cdots + \alpha_r^\vee,\beta}$ for all simple roots $\beta$:
in other words, $\alpha_1^\vee + \cdots + \alpha_r^\vee$ is expressed as an integer
combination of fundamental coweights.
In the last column, we express the smallest root $\theta$ such that
$\scal{\alpha_1^\vee + \cdots + \alpha_r^\vee,\theta}=2$.

\begin{table}[ht]
$$
\begin{array}{|c|c|c|c|}
 \hline
 \multirow{2}{*}{$(G,\varpi)$} & \multirow{2}{*}{Condition} &
 \multirow{2}{*}{$\alpha_1^\vee+\ldots+\alpha_r^\vee$} & \multirow{2}{*}{$\theta$} \\
 &&&\\
 \hline
 (A_{p+q-1},\varpi_{p-1}) & \begin{array}{c}r<p \\ \mbox{ or } r<q \end{array} & 0 \cdots 0 1 0 \cdots 0 1 0 \cdots 0 &
 0 \cdots 0 1 1 \cdots 1 1 0 \cdots 0 \\
 \hline
 \multirow{2}{*}{$(A_{p+q-1},\varpi_{p-1})$} & \multirow{2}{*}{$r=p=q$} & \multirow{2}{*}{$0 \cdots 0 2 0 \cdots 0$} &
 \multirow{2}{*}{$0 \cdots 0 1 0 \cdots 0$} \\
&&& \\
 \hline
 \multirow{2}{*}{$(B_n,\varpi_1)$} & \multirow{2}{*}{$r=1$} & \multirow{2}{*}{$0 1 0 \cdots 0$} & \multirow{2}{*}{$1 2 \cdots 2 2$} \\
 &&& \\
 \hline
 \multirow{2}{*}{$(B_n,\varpi_1)$} & \multirow{2}{*}{$r=2$} & \multirow{2}{*}{$2 0 0 \cdots 0$} & \multirow{2}{*}{$1 0 \cdots 0 0$} \\
 &&& \\
 \hline
 \multirow{2}{*}{$(C_n,\varpi_n)$} & \multirow{2}{*}{$r < n$} & \multirow{2}{*}{$0 \cdots 0 1 0 \cdots 0 0$} & \multirow{2}{*}{$0 \cdots 0 2 2\cdots 2 1$} \\
 &&& \\
 \hline
 \multirow{2}{*}{$(C_n,\varpi_n)$} & \multirow{2}{*}{$r = n$} & \multirow{2}{*}{$0 \cdots 0 2$} & \multirow{2}{*}{$0 \cdots 0 1$} \\
 &&& \\
 \hline
 (D_n,\varpi_1) & r = 1 & 010 \cdots 0\cdots 0 \begin{array}{c}0\\ 0\end{array} &
 122 \cdots 2\cdots 2 \begin{array}{c}1\\ 1\end{array} \\
 \hline
 (D_n,\varpi_1) & r = 2 & 200 \cdots 0\cdots 0 \begin{array}{c}0\\ 0\end{array} &
 100 \cdots 0\cdots 0 \begin{array}{c}0\\ 0\end{array}\\
 \hline
 (D_n,\varpi_n) & 2r \leq n-2 & 0 \cdots 010\cdots 0 \begin{array}{c}0\\ 0\end{array} &
 0 \cdots 122\cdots 2 \begin{array}{c}1\\ 1\end{array} \\
 \hline
 (D_n,\varpi_n) & 2r = n-1 & 0 \cdots 000\cdots 0 \begin{array}{c}1\\ 1\end{array} &
 0 \cdots 000\cdots 1 \begin{array}{c}1\\ 1\end{array} \\
 \hline
 (D_n,\varpi_n) & 2r = n & 0 \cdots 000\cdots 0 \begin{array}{c}2\\ 0\end{array} &
 0 \cdots 000\cdots 0 \begin{array}{c}1\\ 0\end{array} \\
 \hline
 (E_6,\varpi_1) & r = 1 & \poidsesixnu 000001 &\poidsesixnu 123212 \\
 \hline
 (E_6,\varpi_1) & r = 2 & \poidsesixnu 100010 & \poidsesixnu 111110 \\
 \hline
 (E_7,\varpi_7) & r = 1 & \poidseseptnu 1000000 & \poidseseptnu 2343212 \\
 \hline
 (E_7,\varpi_7) & r = 2 & \poidseseptnu 0000100 & \poidseseptnu 0122211 \\
 \hline
 (E_7,\varpi_7) & r = 3 & \poidseseptnu 0000020 & \poidseseptnu 0000010\\
 \hline
\end{array}
$$
\caption{Dominant orthogonal sequences}
\label{equa:tableau}
\end{table}

\begin{rema}
 \label{rema:tube-type} {\em Tube type cominuscule modules.} The symmetric space $G_\R /
 K_\R$ has tube type for $(A_{p+q-1},\varpi_p)$ 
 when $p=q$, for $(B_n,\varpi_1)$, for $(C_n,\varpi_n)$, for $(D_n,\varpi_n)$ when $n$ is even, for
 $(D_n,\varpi_1)$, and for $(E_7,\varpi_7)$. Glancing at Table~\ref{equa:tableau}, one can readily
 check that $G_\R / K_\R$ has tube type if and only if $z=\alpha_1^\vee + \cdots
 + \alpha_p^\vee$. 

 By abuse of notation, we will say that the $G$-module $\E$ itself has tube type if the corresponding symmetric
 space $G_\R / K_\R$  has tube type. 

 If $\E$ has tube type it follows that $\scal{\suma{p},\beta} = 2 \delta_{\beta,\zeta}$ because by
 definition  $z$ is equal to $2\varpi_{\zeta^\vee}$. Since $\zeta$ and all the roots $\alpha_i$ are long, we
 have $\scal{\av_i,\zeta} = \scal{\alpha_i,\zeta^\vee}$, so $\alpha_1 + \cdots + \alpha_p = 2 \varpi$.
 We get $\varpi- (\alpha_1 + \cdots + \alpha_p) = -\varpi$ and this is the lowest weight
 $\mu_0$ by Proposition \ref{prop:cascade}(4). This implies that 
 $\scal{\mu_0,z}=-\scal{\varpi,z}$, so $\zmax = p$ and $\E_p=\E_{\mu_0}$. Since the lowest weight is the
 opposite of the highest weight, the tube type representation $\E$ is autodual: $\E \simeq \E^\vee$. Moreover,
 for any weight $\chi$, we have an isomorphism of $T$-modules $\E_{-\chi} \simeq
 \E_\chi^\vee$. Since the characters of $K$ inject into the characters of $T$, we get in particular
 that $\E_p=\E_{-\varpi}\simeq \E_{\varpi}^\vee=\E_0^\vee$ as $K$-modules.
\end{rema}

We make the following observations:

\begin{prop}
 \label{prop:h}
Let $(\a_1,\ldots,\a_r)$ be a dominant orthogonal sequence and $h=\alpha_1^\vee + \cdots +
\alpha_r^\vee$. Then:
 \begin{itemize}
  \item $h$ is dominant: for any positive root $\alpha$, we have $\scal{\alpha,h} \geq 0$.
  \item For any root $\alpha$, we have $\scal{\alpha,h} \leq 2$.
  \item If a root $\alpha$ satisfies $\scal{\alpha,h} = 2$, then $\scal{\alpha,z}=2$.
  \item We have $\scal{\varpi,h}=r$.
 \end{itemize}
\end{prop}
\begin{proof}
 Recall the table \ref{equa:tableau}. The first point has been proved in Proposition \ref{prop:cascade}(3).
 Let $\Theta$ be the highest root of the root system of $G$,
 which can be found for example in \cite{bourbaki}.
 Since $h$ is dominant, the second item follows from the
 fact that $\scal{\Theta,h}=2$ in all cases.
 For the third item, we have indicated in the last column of the array the smallest root $\theta$ such that
 $\scal{\theta,h}=2$. It is thus enough to check that $\scal{\theta,z}=2$, or equivalently that $\theta$
 has coefficient $1$ on the simple root corresponding to $\varpi$. This is again readily checked.

 To prove that $\scal{\varpi,h}=r$, recall that if $(\a_1,\ldots,\a_p)$ is the maximal dominant
 orthogonal sequence, the weight
 $\varpi- \sum_1^p \alpha_i$ is the lowest weight $\mu_0$ by Proposition \ref{prop:cascade}(4). Thus,
 $\scal{\varpi-\mu_0,h} = \scal{\alpha_1+\cdots+\alpha_p,\alpha_1^\vee+\cdots+\alpha_r^\vee}=2r$,
 since $\scal{\alpha_i,\alpha_j^\vee}=2\delta_{i,j}$ by orthogonality of the roots $\alpha_i$.
 Moreover, $h$ is part of some $\fsl_2$-triple $(x,h,y)$, so
 $\scal{\mu_0,h}=-\scal{\varpi,h}$, by the representation theory of $\fsl_2$.
 This proves that $\scal{\varpi,h}=r$.
\end{proof}

\begin{rema}
 The four points of the proposition fail in most cases if we perform the algorithm starting with a
 weight which is not cominuscule.
\end{rema}

\subsection{The submodule associated with a nilpotent element}
\label{subsec:higgs-submodule}\hfill\medskip

We explain now that an element $y\in\fu_-$ defines a graded subspace $\V$ in $\E$.
From now on we assume that $y=y_{\a_1}+ \cdots + y_{\a_r}$, where $(\a_1,\ldots,\a_r)$ is the dominant
orthogonal sequence obtained by the algorithm of Definition~\ref{def:cascade} and  we denote by $h$ the element
$\alpha_1^\vee + \cdots + \alpha_r^\vee$ explicited in Table \ref{equa:tableau}.

\medskip

We begin with the following consequence of Proposition \ref{prop:Ei}:

\begin{coro}
 \label{coro:chi-Ei}
 For $\chi \in X(\E_i)$, we have $\scal{\chi,h} \geq r - 2i$.
\end{coro}
\begin{proof}
 This follows from Proposition \ref{prop:Ei} and the fact that for any root $\alpha$,
 we have $\scal{\alpha,h} \leq 2$ (Proposition \ref{prop:h}).
\end{proof}

We define the subspace $\V\subset\E$ associated to $y$ by the equality condition in the
inequality of this last corollary: 

\begin{defi}
\label{defi:V}
 Let $\V \subset \E$ be defined by $\V = \oplus \V_i$ with
 $$\V_i = \bigoplus_{\chi \in X(\E_i) : \scal{\chi,h}=r-2i} \E_\chi \ .$$
\end{defi}
Observe that the subspaces $\V_i$ are non trivial exactly when $i \in \{0,\ldots,r\}$, since by
Proposition \ref{prop:h}, we have
$\scal{\chi,h} \in \{-r,-r+1,\ldots,r\}$ for $\chi \in X(\E)$.

\def\bbF{\mathbb F}

The elements $y$ and $h$ fit in a $\fsl_2$-triple $(x,h,y)$ for some $x\in\fu_+$. Define a descending filtration
$(\bbF_k)_{r\geq k\geq -r}$ of
$\E$ by
$$
\bbF_k:=\bigoplus_{\chi \in X(\E) : \scal{\chi,h}\leq k}\E_\chi.
$$
The centralizer of $y$ stabilizes each subspace $\bbF_k$, and since any two $\fsl_2$-triples
are congruent under this centralizer \cite{mcgovern}*{Theorem 3.8}, this filtration only depends on $y$, and not on
$h$: see also the argument given for the canonical parabolic subalgebra at the end of
\cite{mcgovern}*{Paragraph 3.2}. Using the representation theory of $\fsl_2$, we can see that $\bF_k$ depends only
on $y$, since it can be described as follows:
\begin{equation}
\label{equa:filtration-y}
\bbF_k=\sum_{\begin{array}{c}
       \ell \geq 0 \\
       k+\ell+1 \geq 0
      \end{array}}
\Ker y^{k+\ell+1} \cap \Im y^\ell
\end{equation}

Moreover, for all $k\geq 0$,
$$
y^k:\bbF_k/\bbF_{k-1}\stackrel{\sim}{\longrightarrow}\bbF_{-k}/\bbF_{-k-1}
$$
(where $\bbF_{-r-1}:=\{0\}$).

It is plain from Corollary~\ref{coro:chi-Ei} that the $\V_i$'s can alternatively be defined by
$\V_i=\E_i\cap\bbF_{r-2i}$, for all $i=0,1,\ldots,r$.

This implies the following equality on dimensions:

\begin{lemm}
 \label{lemm:dim-Vi}
 We have $\dim \V_{r-i} = \dim \V_i$.
\end{lemm}
\begin{proof}
 We know that $\V_i=\E_i\cap\bbF_{r-2i}$,
$\V_{r-i}=\E_{r-i}\cap\bbF_{2i-r}$, and that for $0\leq i\leq r/2$, $y^{r-2i}$ is an isomorphism between
$\bbF_{r-2i}/\bbF_{r-2i-1}$ and $\bbF_{2i-r}/\bbF_{2i-r-1}$. By Proposition~\ref{prop:Ei},
$y^{r-2i}$ maps $\E_i$ to $\E_{r-i}$. Since $\E_k\cap\bbF_{r-2k-1}=\{0\}$ for all $k$ by
Corollary~\ref{coro:chi-Ei}, we get that $y^{r-2i}$ is an isomorphism between $\V_i$ and $\V_{r-i}$.
\end{proof}

\medskip

The following algebraic fact is at the heart of the construction of Higgs subsheaves in Section~\ref{sec:MW}:

\begin{prop} \label{prop:algHiggs}
 The following inclusions hold:
 \begin{itemize}
  \item $y \cdot \V_i \subset \V_{i+1}$.
  \item $\fu_+ \cdot \V_i \subset \V_{i-1}$.
 \end{itemize}
\end{prop}
\begin{proof}
 The first statement holds because $y \in \bigoplus_{i} \fg_{-\alpha_i}$ and each $\alpha_i$ satisfies
 $\scal{-\alpha_i,z}=\scal{-\alpha_i,h}=-2$. The second statement holds because for $\alpha$ a root of
 $\fu_+$, we have $\scal{\alpha,z}=2$ (and $\scal{\alpha,h} \leq 2$ by Proposition \ref{prop:h}).
\end{proof}

\begin{defi}\label{defi:qlh}
 Let $\fq \subset \fg$ be the parabolic subalgebra defined by the coweight $z-h$:
 $$ \fq = \ft \oplus \bigoplus_{\alpha:\scal{\alpha,h} \leq \scal{\alpha,z}} \fg_\alpha \ . $$
A Levi factor of $\fq$ is
$$\ds \fl=\ft \oplus \bigoplus_{\alpha:\scal{\alpha,h} = \scal{\alpha,z}} \fg_\alpha \ . $$
Let  $\fl_\pm \subset \fl$ be the nilpotent subalgebras of $\fl$ defined by
$$\ds \fl_\pm = \bigoplus_{\alpha:\scal{\alpha,h} = \scal{\alpha,z} = \pm 2} \fg_\alpha \ . $$
A Levi factor of $\fk\cap\fq$ is
$$\ds \fh = \ft \oplus \bigoplus_{\alpha:\scal{\alpha,h} = \scal{\alpha,z}=0} \fg_\alpha \ .$$

We denote by $Q$, resp. $H$, the subgroups of $G$ with Lie algebra $\fq$, resp. $\fh$. We observe that
$H$ is a Levi subgroup of $K \cap Q$: in fact, we have $K\cap Q = H \cdot R(K \cap Q)$, where
$R(K \cap Q)$ denotes the radical of $K \cap Q$.
\end{defi}

For the convenience of the reader, the conditions that define the different Lie subalgebras of $\fg$
 we are 
considering  are displayed in Table~\ref{equa:tableau-groupes} (we abbreviate the condition
$\scal{\alpha,z}=0$, 
resp. $\scal{\alpha,h}=0$, on a root $\alpha$ as $z=0$, resp. $h=0$):


\begin{table}[ht]
$$
\begin{array}{|l|c|c|c|c|c|c|}
 \hline \mbox{Lie algebra} & \fk & \fu_\pm & \fq & \fl & \fl_\pm & \fh \\
 \hline \mbox{Condition} & z=0 & z = \pm 2 & h \leq z & h=z & h=z=\pm 2 & h=z=0 \\
 \hline
\end{array}
$$
\caption{Lie subalgebras under consideration}
\label{equa:tableau-groupes}
\end{table}

We have the following lemmas concerning the $\V_i$'s:
\begin{lemm}
 \label{lemm:Vi+1}
 We have $\V_{i+1} = \fl_- \cdot \V_i$ and $\V_{i-1} = \fl_+ \cdot \V_i$.
\end{lemm}
\begin{proof}
Let us denote by $X(\E_i)$ the set of weights
of $\E_i$. We know by Proposition \ref{prop:Ei} that $\E_{i+1} = \fu_- \E_i$. Thus,
$$
\E_{i+1} = \left ( \bigoplus_{\alpha:\scal{\alpha,z} = -2} \fg_\alpha \right ) \cdot
\left ( \bigoplus_{\chi \in X(\E_i)} \E_\chi \right ) =
\bigoplus_{\chi \in X(\E_i), \alpha \in \Phi(\fu_-) } \E_{\chi+\alpha} \ .
$$
Given $\chi \in X(\E_i)$ and $\alpha \in \Phi(\fu_-)$, we can make two observations:
\begin{itemize}
 \item If $\scal{\chi,h} > r - 2i$ then we have $\scal{\chi+\alpha,h} > r -2(i+1)$ and thus
 $\E_{\chi + \alpha} \not \subset \V_{i+1}$.
 \item If $\scal{\alpha,h} > -2$ then $\scal{\chi+\alpha,h} > r -2(i+1)$.
\end{itemize}
The first equality of the lemma follows from these observations.
The proof of the second equality is similar.
\end{proof}

Note that $\fl_-$ is an $\fh$-module. More precisely, we have:

\begin{lemm}
 \label{lemm:Vi-irr}
 The modules $\V_i$ are irreducible $\fh$-modules.
 The Lie algebra $\fl_-$ is an irreducible $\fh$-module.
\end{lemm}
\begin{proof}
Denote by $\fk_+$ resp. $\fh_+$ the sum of the root spaces for positive roots in $\fk$ resp. $\fh$.
Let $i$ be fixed and such that $\V_i \neq \{0\}$. By Proposition \ref{prop:Ei}, $\E_i$ is an irreducible $\fk$-module.
Let $\mu_i$ be the lowest weight of $\E_i$. We have $\E_{\mu_i} \subset \V_i$.
Since $\E_i$ is irreducible, we have $\E_i = U(\fk_+) \cdot \E_{\mu_i}$ (here $U$ denotes the universal envelopping algebra).
Now, as in the proof
of the Lemma \ref{lemm:Vi+1}, we see that this implies that $\V_i = U(\fh_+) \cdot \E_{\mu_i}$. This proves that $\V_i$ is irreducible.
Now, by Lemma \ref{lemm:Vi+1} again, we have $\V_1 = \fl_- \cdot \V_0 \simeq \fl_- \cdot \E_0$: thus
$\fl_-$ is also irreducible.
\end{proof}

This allows to compute the {\em slope} of the $H$-modules $\V_i$. We first need two definitions and a lemma.

\begin{defi}
 \label{defi:equislope}
 Let $\W$ be a $H$-module (here $H$ can be any reductive group).
 The {\em slope} of $\W$ is the element $\mu(\W)=\frac{\det(\W)}{\dim(\W)}$ in $X(H) \otimes_\Z \Q$, where
 $X(H)$ is the character group of $H$.
 We say that $\W$ is {\em equislope} if, for the decomposition $\W=\bigoplus_i \W_i$ as a sum
 of irreducible submodules, we have
 $\forall i,j$, $ \mu(\W_i)=\mu(\W_j)$.
\end{defi}

\begin{lemm}
 \label{lemm:equislope}
 Let $\W$, $\W'$ be equislope $H$-modules (here again, $H$ can be any reductive group). Then $\W
 \otimes \W'$ is equislope.
\end{lemm}
\begin{proof}
Let $Z$ be the center of $H$. It is known that restriction to $Z$ yields an injection $X(H)
\hookrightarrow X(Z)$. In fact, by \cite{borel}*{Proposition 14.2},
 we have $H=Z \cdot (H,H)$, and any character of $H$ is trivial on $(H,H)$.

 Let us first assume that $\W$ and $\W'$ are irreducible $H$-modules. By Schur's lemma, there are
 characters $\chi,\psi$ of $Z$ such that $\forall g \in Z, \forall w \in \W, \forall w' \in \W', g
 \cdot w = \chi(g)w$
 and $g \cdot w' = \psi(g)w'$. Therefore, in $X(Z) \otimes \Q$, we have
 $\chi = \mu(\W)_{|Z}$ and $\psi = \mu(\W')_{|Z}$.
 Let now $\mathbb U \subset \W \otimes \W'$ be an irreducible component. We have
 $\forall u \in \mathbb U, \forall g \in Z, g\cdot u = \chi(g) \psi(g) u$.
 Therefore $\mu(\mathbb U)_{|Z} = \chi + \mu$. Thus, ${(\mu(\W)+\mu(\W'))}_{|Z} = \mu(\mathbb U)_{|Z}$.
 Since restriction of characters to $Z$ is injective, we have $\mu(\mathbb U) = \mu(\W) + \mu(\W')$.

 Let now $\W$ and $\W'$ be arbitrary $H$-modules, and write the decomposition into irreducible submodules:
 $\W=\bigoplus \W_i$ and $\W'=\bigoplus \W'_j$. Let $i,j$ be fixed and let $\mathbb U \subset \W_i
 \otimes \W_j'$ be an
 irreducible factor. We have proved that $\mu(\mathbb U)=\mu(\W_i)+\mu(\W'_j)$. Since $\W$ and $\W'$
 are equislope,
 we have $\mu(\W_i)=\mu(\W)$ and $\mu(\W'_j)=\mu(\W')$. Thus, $\mathbb U$ has slope
 $\mu(\W)+\mu(\W')$ and the lemma is proved.
\end{proof}

\begin{rema}
 Let $\Gamma$ be the Schur functor associated with a partition $\lambda$ of weight $|\lambda|$. Let $\W$
 be an $H$-module. Since
 $\Gamma(\W)$ is a direct factor of $\W^{\otimes |\lambda|}$, it follows from the Proposition
 that $\mu(\Gamma(\W)) = |\lambda| \mu(\W)$.
\end{rema}

We apply this lemma to the $H$-modules $\V_i$:

\begin{prop}
 \label{prop:calcul-pente}
 We have $\mu(\V_i) = \mu(\V_0)+i\mu(\V_0^\vee \otimes \V_1)$.
\end{prop}
\begin{proof}
 The $H$-modules $\V_i$ are irreducible by Lemma \ref{lemm:Vi-irr}, thus equislope.
 The same holds for the $H$-module $\fl_-\simeq \V_0^\vee \otimes \V_1$, and we have
 $\mu(\fl_-) = \mu(\V_0^\vee \otimes \V_1)$. Let $i$ be fixed.
 By Lemma \ref{lemm:equislope}, $\V_i \otimes \fl_-$ is equislope. Since, by Lemma \ref{lemm:Vi+1},
 $\V_{i+1}$ is a direct factor of $\V_i \otimes \fl_-$, it follows that $\mu(\V_{i+1}) = \mu(\V_i) + \mu(\fl_-)$.
\end{proof}

Concerning the submodule $\V$, we have

\begin{prop} \label{prop:stabV}
\begin{enumerate}
\item\label{item:parabolic}
 The subspace $\V \subset \E$ is stable under $\fq$.
\item $\V_i$ is a $K\cap Q$-module.
\item \label{item:reductive-Q}
 The nilpotent radical of $\fq$ acts trivially on $\V$.
\item \label{item:stabilizer-V}
 The subgroup of elements in $G$ which preserve $\V$ is exactly $Q$.
\end{enumerate}
\end{prop}

\begin{proof}

Let us prove the first point.  The subspace $\V$ is clearly stable under $\ft$. Let $\alpha$ be such
that $\scal{\alpha,h} \leq \scal{\alpha,z}$.
 Let $v \in \V_\chi \subset \V_i$, thus we have $\scal{\chi,h} = r-2i$.
 For $x \in \fg_\alpha$, we have $x \cdot v \in \E_{\chi+\alpha} \subset \E_{i-\scal{\alpha,z}/2}$. Since
 $\scal{\alpha,h} \leq \scal{\alpha,z}$, we have $\scal{\chi + \alpha , h} \leq r-2i+\scal{\alpha,z}$,
 thus either $x \cdot v = 0$ or $\scal{\chi + \alpha , h} = r-2i+\scal{\alpha,z}$, by Corollary \ref{coro:chi-Ei}.
 In the second case, we get $x \cdot v \in \V_{i-\scal{\alpha,z}/2}$.
The fact that $\V_i$ is a $K\cap Q$-module follows because $\V_i = \E_i \cap \V$, $\E_i$ is $K$-stable, and
$\V$ is $Q$-stable.

For the third point, let $\alpha$ be a root of the nilpotent radical of $\fq$: we have $\scal{\alpha,h} < \scal{\alpha,z}$.
 Let $v \in \V_\chi \subset \V_i$, thus we have $\scal{\chi,h} = r-2i$.
 For $x \in \fg_\alpha$, we have $x \cdot v \in \E_{\chi+\alpha} \subset \E_{i-\scal{\alpha,z}/2}$.
 However, we get $\scal{\chi + \alpha , h} > r-2i+\scal{\alpha,z}$. Thus, $x \cdot v = 0$.

Finally, to prove that the stabilizer of $\V$ is exactly $Q$, let us denote by
$\mathfrak{stab}(\V) \subset \fg$ the Lie subalgebra
preserving the subspace $\V$ in $\E$.
 We know by (\ref{item:parabolic}) that $\mathfrak{stab}(\V) \supset \fq$.
On the other hand, let $\alpha$ be a root and let $0 \neq x \in \fg_\alpha$ be such that
 $x \cdot \V \subset \V$.
 Let us assume as a first case that $\scal{\alpha,z}=-2$. Since, by Proposition \ref{prop:Ei}(e),
 the action of $\fg$ on $\E$ induces
 an isomorphism $\E_1 \simeq \E_0 \otimes \fu_-$, we have $x \cdot \E_{\varpi} = \E_{\varpi+\alpha}$.
 Then $\E_{\varpi+\alpha} \subset \V$, so $\scal{\alpha,h}=-2$, and $x \in \fq$.
 Assume now that $\scal{\alpha,z}=0$, and by contradiction that $\scal{\alpha,h}>0$.
 Proposition \ref{prop:h} then implies that $\scal{\alpha,h}=1$.
 For any integer $i$,
 we cannot have $\scal{\alpha,\alpha_i^\vee}=2$ because this would imply $\alpha=\alpha_i$ and
 $\scal{\alpha,h}=2$. Let $i$ be such that $\scal{\alpha,\alpha_i^\vee}>0$: then
 $\scal{\alpha,\alpha_i^\vee}=1$. Therefore, $\alpha-\alpha_i$ is a root.
 Since $\scal{\alpha,z}=0$, $\fg_\alpha \cdot \E_0 \subset \E_0 \cap \E_{\varpi+\alpha}$, thus
 $\fg_\alpha \cdot \E_0 = \{0\}$. It follows that
 $$
 \fg_{\alpha-\alpha_i} \cdot \E_{\varpi} = [ \fg_\alpha , \fg_{-\alpha_i} ] \cdot \E_{\varpi}
 = \fg_\alpha \cdot \left ( \fg_{-\alpha_i} \cdot \E_{\varpi} \right ).
 $$
 Since $\fg_{\alpha-\alpha_i} \cdot \E_{\varpi} = \E_{\varpi + \alpha - \alpha_i}$
 (again by by Proposition \ref{prop:Ei}(e)), this contradicts $x \cdot \V \subset \V$.

 Let now ${\rm Stab}(\V) \subset G$ be the subgroup stabilizing $\V$. We have $Q \subset {\rm Stab}
 (\V)$, thus
 ${\rm Stab}(\V)$ is parabolic and therefore connected \cite{humphreys}*{Corollary 23.1.B}. It
 follows that $Q={\rm Stab}(\V)$.
\end{proof}

Moreover:

\begin{prop}
 \label{prop:V-irr}
 The $\fl$-module $\V$ is irreducible and, as an $\fl$-module, it has tube type.
\end{prop}
\begin{proof}
 Combining Lemmas \ref{lemm:Vi+1} and \ref{lemm:Vi-irr}, we get the irreducibility of $\V$.
 Note that, by definition, any root $\alpha$ of $\fl$ satisfies $\scal{\alpha,z} = \scal{\alpha,h}$.
 Now, a dominant orthogonal sequence for the weight $\varpi$ as a weight of $\fg$ is also a dominant
 orthogonal sequence for  $\varpi$ as a weight of $\fl$. It follows that $\V$ satisfies the 
 assumption of Remark \ref{rema:tube-type}, and therefore  $\V$ has tube type. This means
   that if $L$ is the subgroup of $G$ with Lie algebra $\fl$ and if we set $L_\R=L\cap G_\R$ and
   $H_\R=H\cap G_\R=L_\R\cap K_\R$, then the
   Hermitian symmetric space $L_\R/H_\R$ has tube type.
\end{proof}

\commentaire{
\begin{coro}
 \label{coro:dim-Vi}
 We have $\dim \V_{r-i} = \dim \V_i$.
\end{coro}
\begin{proof}
Indeed, it is well-known that a representation of tube type carries a non-degenerate symplectic form \pe{Référence}. This implies that
$\V$ is self-dual, and therefore $\chi \in X(\V)$ if and only if $-\chi \in X(\V)$. This proves the corollary.
\end{proof}
}

\medskip

\begin{exem}
We give an example of this construction.
We assume that we are in the first arrow of the array (\ref{equa:tableau}), namely that $G$ has
type $A_{p+q-1}$ for some positive integers $p\leq q$ and that $\varpi = \varpi_{p-1}$.

Let
$y\in\fu_-$ be an element of rank $r$.
We have a natural decomposition $\C^{p+q}=N \oplus A \oplus I \oplus B$, with
$\ker y = N \oplus I \oplus B$ and $\Im y = I$. We choose a basis $(e_i)$ such that
$N$, resp. $A$, $I$, $B$ is generated by $(e_1,\ldots,e_{p-r})$, resp.
$(e_{p-r+1},\ldots,e_p)$, $(e_{p+1},\ldots,e_{p+r})$, $(e_{p+r+1},\ldots,e_{p+q})$.

The element $h \in \ft$ acts on the Lie algebra of $G$, and the corresponding weights of block matrices
from $N \oplus A \oplus I \oplus B$ to itself are
$
\left (
\begin{array}{cccc}
0 & -1 & 1 & 0 \\
1 & 0 & 2 & 1 \\
-1 & -2 & 0 & -1 \\
0 & -1 & 1 & 0
\end{array}
\right )
$.
Beware that with our choice of base $(e_i)$, the positive roots do \emph{not} correspond to matrix coordinates
above the diagonal.

The weights for the central element $z$ are
$
\left (
\begin{array}{cccc}
0 & 0 & 2 & 2 \\
0 & 0 & 2 & 2 \\
-2 & -2 & 0 & 0 \\
-2 & -2 & 0 & 0
\end{array}
\right )
$.
Thus the Lie algebras $\fq$, $\fl$ and $\fh$  are respectively
$$
\left (
\begin{array}{cccc}
* & * & * & * \\
0 & * & * & * \\
0 & * & * & * \\
0 & 0 & 0 & *
\end{array}
\right ),\quad
\left (
\begin{array}{cccc}
* & 0 & 0 & 0 \\
0 & * & * & 0\\
0 & * & * & 0 \\
0 & 0 & 0 & *
\end{array}
\right )\ \mbox{ and } \
\left (
\begin{array}{cccc}
* & 0 & 0 & 0 \\
0 & * & 0 & 0\\
0 & 0 & * & 0 \\
0 & 0 & 0 & *
\end{array}
\right ).
$$
Thus we see that $Q$ is exactly the stabilizer of the flag $(N \subset N \oplus A \oplus I)$, that
the subgroup of $G$ corresponding to $\fl$ is ${\rm S}(\GL(N)\times\GL(A\oplus I)\times\GL(B))$ and
$H$ is the block diagonal group ${\rm S}(\GL(N) \times \GL(A) \times \GL(I) \times \GL(B))$, and
finally that the intersections of $G_\R=\SU(p,q)$ with these two latter groups are  isomorphic to
${\rm S}(\U(p-r)\times\U(r,r)\times\U(q-r))$ and ${\rm S}(\U(p-r) \times \U(r) \times \U(r) \times
\U(q-r))$ respectively.

On the other hand, we have $\E=\wedge^p(N \oplus A \oplus I \oplus B)$ and it is easy to check that
$\V=\wedge^{p-r}N \otimes \wedge^r(A \oplus I)$, which confirms that the stabilizer of $\V$ is $Q$.

\begin{rema}\label{rem:dual}
    Given an element $y'$ in $\fu_+$ instead of $\fu_-$, one can consider the dual
    representation $\E^\vee$ of $G$, and construct as above a graded subspace
    $\V'=\oplus_{i=0}^{\rk y'}\,\V'_i$ of $\E^\vee =\oplus_{i=0}^{p}\,\E_i^\vee$
    (with $\V_i' \subset \E_i^\vee)$. It has the
    same properties as the subspace $\V\subset\E$ discussed above with the obvious modifications,
    e.g. the statement of
    Proposition~\ref{prop:algHiggs} for $\V'$ is that $y'\cdot\V_i'\subset \V'_{i+1}$ and
    $\fu_-\cdot\V'_i\subset\V'_{i-1}$.
  \end{rema}




\end{exem}

\section{Higgs bundles on foliated Kähler manifolds}
\label{sec:HiggsFoliated}

\subsection{Harmonic Higgs bundles}
\label{subsec:Higgs}\hfill\medskip

We keep the notation of the previous sections. In particular, $G_\R$ is a simple noncompact
Hermitian Lie group, $K_\R$ its maximal compact subgroup, $M=G_\R/K_\R$ the associated irreducible
Hermitian symmetric space of the noncompact type, $G$ and $K$ are the complexifications of $G_\R$
and $K_\R$, and  $\gg=\kg\oplus\ug$ and $\gg_\R=\kg_\R\oplus\ug_\R$ are the
associated Cartan decompositions of the Lie algebras of $G$ and $G_\R$.

\medskip

Let now $Y$ be a compact Kähler manifold, $\G=\pi_1(Y)$ its fundamental group, and $\r:\G\pfd G_\R$ be a
reductive representation (group homomorphism) of $\G$ into  $G_\R$.

In this case, by \cite{CorletteFlatGBundles},
there exists a $\r$-equivariant {\em harmonic} map $f$ from the universal cover $\tilde Y$ of $Y$ to the
symmetric space $M=G_\R/K_\R$ associated to $G_\R$.
The fact that $Y$ is Kähler and the nonpositivity of the complexified sectional curvature of $M$
imply by a Bochner formula due to~\cites{Sampson,Siu} that the map $f$ is {\em pluriharmonic} (i.e. its
restriction to 1-dimensional complex submanifolds of $Y$ is still harmonic), and that the image of its
(complexified) differential at every point $y\in Y$ is an Abelian subalgebra of $T^\C_{f(y)} M$
identified with $\ug$.

\medskip

By the work of C. Simpson, this
gives a {\em harmonic $G_\R$-Higgs principal bundle} $(P_K,\t)$ on $Y$. We will now briefly describe the
construction and the properties of such a Higgs bundle. Details and proofs can be found in the
original papers~\cites{S1,S2}.

There is a flat principal
bundle $P_{G_\R}\pfd Y$ of group $G_\R$ associated to the representation $\r$. The $\r$-equivariant map
$f:\tilde Y\pfd G_\R/K_\R$ defines a reduction of its structure group to $K_\R$, i.e. a principal
$K_\R$ bundle $P_{K_\R}\subset P_{G_\R}$. The differential of $f$ can be seen as a 1-form with
values in $P_{K_\R}(\ug_\R)$, the vector bundle associated to the adjoint action of $K_\R$ on
$\ug_\R$.

If we enlarge the structure group of $P_{K_\R}$ to $K$ and complexify the whole situation,
the pluriharmonicity of $f$ implies that the $K$-principal bundle $P_K\pfd Y$ is a holomorphic
bundle. Moreover, the  (complexified) differential
$d^{1,0}f:T^{1,0}Y\pfd T^\C M$
of the harmonic map $f$ defines a holomorphic section $\t$ of
$P_K(\ug)\otimes\Omega^1_Y$, where $P_K(\ug)$ is the holomorphic vector bundle associated to the
principal bundle $P_K$ and the adjoint representation of $K$ on $\ug$. The section $\t$ is called
the {\em Higgs field} and satisfies the integrability
condition $[\t,\t]=0$ as a section of $P_K(\ug)\otimes\Omega^2_Y$.  The couple $(P_K,\t)$ is called a {\em $G_\R$-Higgs principal bundle} on $Y$.

If now $\E$ is a (complex) representation of $G$ we can construct the associated holomorphic vector bundle
$E:=P_K(\E)$ over $Y$. Since $\E$ is a representation of $G$ and not only of $K$, the Higgs field
$\t$ can be seen as a holomorphic 1-form with values in the endomorphisms of $E$, i.e. a
section of $\End(E)\otimes\Omega^1_Y$.  The couple $(E,\t)$
is called a {\em $G_\R$-Higgs vector bundle} on $Y$. The harmonic map $f$, seen as a reduction of the
structure group of $P_K$ to the compact subgroup $K_\R$, together with a $K_\R$-invariant metric on
$\E$,  gives a Hermitian metric on $(E,\t)$
called the {\em harmonic metric}.

The existence of this harmonic metric and the fact that $P_K$ comes from a flat principal $G_\R$ bundle imply
that for any representation $\E$ of $G$, the associated Higgs vector bundle $(E,\t)\pfd Y$ is {\em Higgs
polystable} of degree 0, see~\cite{S1}.  To explain
what Higgs polystability means, we first define {\em Higgs subsheaves} of the Higgs bundle $(E,\t)$.
A coherent subsheaf $\F$ of $\cal O_Y(E)$ is a Higgs subsheaf if it is invariant by the Higgs field, i.e. it satisfies
$\t(\F)\subset\F\otimes\Omega^1_Y$. The Higgs vector bundle $(E,\t)$ is said to be Higgs stable if for any Higgs
subsheaf $\F$ of $(E,\t)$ such that $0<\rk\F<\rk E$, we have $\mu(\F)<\mu(E)$, where $\mu(\F)$ is
the slope of $\F$, i.e. its degree (computed w.r.t. the Kähler form of $Y$) divided by its rank. The
Higgs bundle $(E,\t)$ is Higgs polystable if it is a direct sum of Higgs stable Higgs vector bundles
of the same slope.
Note that here $E$ is flat as a $C^\infty$ bundle, so that its degree is zero.

\begin{rema}\label{rem:real}
Since moreover we assumed that $G_\R$ is a Hermitian group, then as a $K$-representation we have
$\ug=\ug_+\oplus\ug_-$
and the Higgs field $\t$ on the principal bundle $P_K$ (or on any associated vector bundle $E$) has two
components $\b\in P_K(\ug_-)\otimes\Omega^1_Y$ and  $\g\in P_K(\ug_+)\otimes\Omega^1_Y$. The
vanishing of one of these components means that the harmonic map $f$ is holomorphic or
antiholomorphic.
\end{rema}

\subsection{Harmonic Higgs bundles on foliated Kähler manifolds}\label{subsec:HiggsFol}\hfill\medskip

Assume now that the base Kähler manifold $Y$ of the harmonic $G_\R$-Higgs vector bundle
$(E,\t)\pfd Y$ of degree 0 admits a smooth holomorphic
foliation by complex curves and that this foliation $\T$ admits an invariant transverse measure.
Our goal in this section is to explain the interplay between the Higgs bundle and the
foliation. Details can be found in~\cite{KM}*{\textsection 2.2}.

We first weaken the notion of Higgs subsheaves of $(E,\t)$ to {\em leafwise Higgs subsheaves} by asking only an
invariance by the Higgs field along the leaves. More precisely  we now  consider the
Higgs field as a section of $\End(E)\otimes L^\vee$, where $L^\vee$ is the dual of the holomorphic
line subbundle $L$ of $T_Y$ tangent of the foliation $\T$. A leafwise Higgs subsheaf $\F$ of
$(E,\t)$ is then a subsheaf of $E$ such that  $\t(\F\otimes L)\subset \F$.

The invariant transverse measure gives a current of integration along
the leaves of $\T$. This allows to define the {\em foliated degree} $\deg_\T\F$ of a coherent sheaf $\F$ on $Y$ by
applying this current to the first Chern class of $\F$.

\medskip

The Higgs bundle enjoys the following leafwise polystability property w.r.t. the foliated degree
(\cite{KM}*{Proposition 2.2}):

\begin{prop}\label{prop:polystab}
  Assume that the invariant transverse measure comes from an invariant transverse volume form. Then the Higgs bundle $(E,\t)$ on the
foliated Kähler manifold $Y$ is {\em weakly polystable along the leaves} in the following sense:
\begin{enumerate}
\item it is {\em semistable along the leaves} of $\T$: if $\F$ is a leafwise Higgs
  subsheaf of $(E,\t)$, then $\deg_\T\F\leq 0$.
 \item if $\F$ is a saturated leafwise Higgs subsheaf of $(E,\t)$ such that $\deg_\T\F=0$, then the
  singular locus $\S(\F)$ of $\F$ is saturated under the foliation $\T$. Moreover, on $Y\backslash
  \S(\F)$, and if $F$ denotes the holomorphic subbundle of $E$ such that $\F=\cal
  O_{Y\backslash\S(\F)}(F)$ and $F^\perp$ its $C^\infty$
  orthogonal complement w.r.t. the harmonic metric, then $\t(F^\perp\otimes L)\subset F^\perp$ and
  for each leaf $\L$ of $\T$ such that $\L\cap\S(\F)=\emptyset$, $F^\perp$ is holomorphic on $\L$
  and $(E,\t)=(F,\t_{|F})\oplus(F^\perp,\t_{|F^\perp})$ is an holomorphic
  orthogonal decomposition on $\L$.
\end{enumerate}
\end{prop}

(In the proposition the {\em singular locus} $\S(\F)$ of a subsheaf $\F$ of $\cal O_Y(E)$ is the
complement of the
subset of $Y$ where $\F$ is the
  sheaf of sections of a subbundle $F$ of $E$. Equivalently, it is the subset of $Y$ where $\cal
  O_Y(E)/\F$ is not locally free. If $\F$ is saturated $\S(\F)$ has codimension at least 2 in $Y$.)

\subsection{The tautological foliation on the projectivized tangent bundle of a complex hyperbolic
  manifold}\label{subsec:toto} \hfill\medskip

A $n$-dimensional complex hyperbolic manifold $X$ is a quotient of the complex hyperbolic $n$-space
$\H=\SU(1,n)/\U(n)$  by a discrete torsion free subgroup $\G$ of $\SU(1,n)$. It is a Hermitian
locally symmetric space of rank 1.

The complex hyperbolic space $\H$ is an open subset in the projective space $\CP^n$: it's the subset of
negative lines in $\C^{n+1}$ for an Hermitian form of signature $(n,1)$. Intersections of lines
of $\CP^n$ with $\H$ are totally geodesic complex subspaces of $\H$ isometric to the Poincar\'e
disc. They are called {\em complex geodesics}. The space $\cal G$ of complex geodesics is the
complex homogeneous space $\SU(1,n)/{\rm S}(\U(1,1)\times \U(n-1))$.

The {\em projectivized tangent bundle} of $\H$ is the complex homogeneous space $\PTH=\SU(1,n)/{\rm
  S}(\U(1)\times\U(1)\times \U(n-1))$. The natural $\SU(1,n)$-equivariant fibration $\pi_{\cal
  G}:\PTH\pfd\cal G$ which associates to a tangent line to $\H$ the complex geodesic it defines is a
disc bundle over $\cal G$.

By $\SU(1,n)$-equivariance, this fibration endows the projectivized tangent bundle
$\PTX=\G\backslash\PTH$ of $X=\Gamma\backslash\H$ with a smooth holomorphic foliation $\T$ by
complex curves, whose
leaves are given by the tangent spaces of the (immersed) complex geodesics in $X$. This foliation is
called the {\em tautological foliation} of $\PTX$ because the tangent line bundle $L$ to the leaves is
naturally isomorphic to the tautological line bundle $\OX$ on $\PTX$.

The space $\cal G$ of complex geodesics of $\H$ is a pseudo-Kähler manifold: it admits a
non-degenerate but indefinite Kähler form $\o_{\cal G}$. This form defines a transverse invariant volume form
for the tautological foliation $\T$ on $\PTX$, and the associated notion of foliated degree
$\deg_\T$ for sheaves on $\PTX$ has the following fundamental property
\cite{KM}*{Proposition~3.1}:
\begin{prop}\label{prop:degT}
Assume that $X=\G\backslash\H$ is compact and let $\pi:\PTX\pfd X$ be the projectivized tangent
bundle of $X$. If
$\F$ is a coherent $\cal O_X$-sheaf, then $\deg_\T(\pi^\star\F)=\deg\F$, where $\deg\F$ is the usual
degree of $\F$ computed w.r.t. the Kähler form on $X$ induced by the $\SU(1,n)$-invariant Kähler
form on $\H$.
\end{prop}

\section{The Milnor-Wood inequality}\label{sec:MW}

Let $X=\G\backslash\H$ be a compact complex hyperbolic manifold of (complex) dimension $n$ and $\r$
be a representation of $\G$ in a simple Hermitian Lie group $G_\R$, whose associated
symmetric space is called $M$. In this section we use
the material developped or recalled in Sections~\ref{sec:subnil} and~\ref{sec:HiggsFoliated}
to prove the Milnor-Wood inequality
$$
|\tau(\rho)|\leq \rk M\,\vol(X).
$$
Recall that we may and do assume that the representation $\rho$ is reductive (see
the discussion just before Assumption~\ref{assu:reductive}).

\subsection{Setup}\hfill\medskip\label{subsec:setup}


\def\tb{\tilde\beta}
\def\tE{\tilde E}
\def\tP{\tilde P}

Consider the representation $\E$ of $G$ defined in \textsection \ref{subsec:minuscule-rep}. As explained
in \textsection\ref{subsec:Higgs}, this gives rise to a flat
harmonic $G_\R$-Higgs vector bundle $(\ovE,\ovt)$ over $X$.
   As a representation of $K$,  $\E=\oplus_{i=0}^p\E_{i}$ where $p$ is the real rank of $G_\R$. This
means that the Higgs bundle $\ovE$ admits the holomorphic decomposition
$\ovE=\oplus_{i=0}^p \ovE_{i}$. Moreover, the components $\ovb\in P_K(\ug_-)$ and $\ovg\in P_K(\ug_+)$ of
the Higgs field $\ovt\in \Hom(\ovE,\ovE)\otimes\Omega^1_X$ (see Remark~\ref{rem:real}) satisfy
$$
\ovb\in\bigoplus_{i=0}^{p-1}\,\Hom(\ovE_{i},\ovE_{i+1})\otimes \Omega^1_X\mbox{ and }
\ovg\in\bigoplus_{i=0}^{p-1}\,\Hom(\ovE_{i+1},\ovE_{i})\otimes \Omega^1_X.
$$

We pull-back the harmonic Higgs bundle $(\ovE,\ovt)\pfd X$ to the projectivized tangent bundle
$\PTX$ of $X$ to obtain a harmonic Higgs bundle that
we denote by $(E,\t)\pfd\PTX$. We restrict the Higgs field $\t$ to the tangent space $L$ of the tautological
foliation on
$\PTX$, so that its components $\b$ and $\g$ satisfy
$$
\b\in\bigoplus_{i=0}^{p-1}\,\Hom( E_{i}, E_{i+1})\otimes L^\vee\mbox{ and }
\g\in\bigoplus_{i=0}^{p-1}\,\Hom( E_{i+1}, E_{i})\otimes L^\vee
$$

\begin{defi}
For $\xi\in\PTX$, the rank $\rk\b_\xi$ of $\b_\xi$ is the largest value of $k$
such that $(\b_\xi)^k:E_0\otimes L^k\pfd E_{k}$  is not zero.

The generic rank $\rk\b$
of $\b$  is the maximum of the ranks of $\b_\xi$ for  $\xi\in\PTX$.

The singular locus of $\b$ is the following subset of $\PTX$:
$$
\S(\b): =\{\xi\in\PTX\,\mid\,
\rk\b_\xi<\rk\b\}=\{\xi\in\PTX\,\mid\,(\b_\xi)^{\rk\b}: E_0\otimes L^r\pfd  E_{r}\mbox{ vanishes}\}
$$

The regular locus of $\b$ is  $\cal R(\b):=\PTX\backslash\S(\b)$.

 The singular locus $\S(\ovb)$ of $\ovb$ is the projection to $X$ of $\S(\b)$.

The regular locus $\cal R(\ovb)$ of $\ovb$ is $X\backslash\S(\ovb)$.
\end{defi}

\begin{defi}
We define similarly $\rk\g$, $\S(\g)$, ${\cal R}(\g)$, $\S(\ovg)$ and
${\cal R}(\ovg)$, except that we consider the dual representation to define the rank
of $\g_\xi$ for $\xi\in\PTX$:  $\rk \g_\xi$ is the largest value of $k$
such that  $({}^t\g_\xi)^k :E_0^\vee\otimes L^k\pfd E_{k}^\vee$  is not zero.
\end{defi}

Observe that while $\S(\b)$ and $\S(\g)$  are analytic subsets of $\PTX$ of codimension at least 1,
$\S(\ovb)$ and $\S(\ovg)$ are analytic subsets of $X$ (because $\pi:\PTX\pfd X$ is a proper map) but
they might be equal to $X$.

\subsection{Rewording of the inequality}\label{subsec:reword}\hfill\medskip

Since the Hermitian symmetric space $M$ associated to
$G_\R$ is a Kähler-Einstein manifold, the first Chern
form $c_1(T_M)$ of its tangent bundle is a constant multiple of the $G_\R$-invariant Kähler form $\o_M$:
$c_1(T_M)=- \frac 1{4\pi}\,c_M\,\o_M$ for some positive constant $c_M$. On the other hand, the line
bundle $\cal L$
associated to the $K$-representation $\E_0$ is
a generator of the Picard group of the compact dual $M^\vee$ of $M$  and it can be checked that the
canonical bundle
$K_{M^\vee}$ of $M^\vee$ is precisely given by ${\cal L}^{\otimes c_M}$, see
e.g. \cite{KMgeneraltype}*{Section 2}.
Therefore the pull-back
$f^\star\o_M$ is $4\pi$ times the first Chern form of the line bundle $f^\star\cal L=\ovE_0$, so that the Toledo
invariant of $\rho$ is
\begin{equation}
 \label{equa:toledo}
 \tau(\rho)=4\pi\,\deg (\ovE_0)=4\pi\,\deg_\T(E_0),
\end{equation}
where the last equality
follows from Proposition~\ref{prop:degT}. Similarly, we get that
\begin{equation}
 \label{equa:deg-canonique}
 \deg(K_X)=\frac{n+1}{4\pi}\vol (X).
\end{equation}

On the other hand, if $L$ is the tangential line bundle to the tautological foliation $\T$ on
the projectivized tangent bundle $\PTX$, and if $L^\vee$ is the dual line bundle, one can compute
as explained in~\cite{KM}*{Section 4.3.1} that
\begin{equation}
 \label{equa:deg-L}
 \deg_\T(L^\vee) = \frac{1}{2\pi} \vol(X).
\end{equation}

Therefore, the Milnor-Wood inequality can be rephrased as an inequality between the foliated degrees
of the line bundles $E_0$ and $L^\vee$ on $\PTX$:
\begin{equation}
\label{equa:rewording}
\left | \deg_\T(E_0) \right |  \leq \frac p2 {\deg_\T(L^\vee)},
\end{equation}
where $p$ is the rank of the symmetric space $M$.


\subsection{Leafwise Higgs subsheaves associated to the components of the Higgs field} \hfill\medskip
\label{subsec:HiggsSubsheaf}

We now define a subsheaf
$\cal V$ of $\cE:=\cal O( E)$ associated associated to
$\b$ in the same way we
defined the submodule $\V$ of $\E$ associated to the nilpotent element $y\in\fu_-$
in Definition \ref{defi:V} (for all $\xi\in L$, $\b(\xi)\in P_K(\fu_-)$ is a nilpotent endomorphism
of the bundle $E$). This subsheaf will be shown to be a leafwise Higgs subsheaf of the Higgs bundle
$(E,\t)$ on $\PTX$. In Section~\ref{subsec:proofMW} this will be used to prove the Milnor-Wood inequality.

More precisely, we follow the alternative definition of $\V$ given after Definition~\ref{defi:V} and
for $k=r,r-1,\ldots,-r+1,-r$, we consider the following subsheaves of $\cE$:
$$
\F_k:=\sum_{\begin{array}{c}
       \ell \geq 0 \\
       k+\ell+1 \geq 0
      \end{array}}
\Ker \b^{k+\ell+1} \cap \Im \b^\ell
$$
where in order to define $\Ker \b^j$ we see $\b$ as a sheaf morphism from $\cE$
to $\cE\otimes {(L^\vee)}^j$ and to define $\Im \b^j$ we see $\b$ as a sheaf
morphism from $\cE\otimes L^j$ to $\cE$.


\smallskip

For $k=0,1,\ldots r$, let $\cal V_k$ be the saturation in $\cE_k:=\cal O( E_k)$ of the
subsheaf  $\cE_k\cap \F_{r-2k}$ and let $\cV=\oplus_{0\leq k\leq r} \cV_k$.

\medskip

Since the sheaves $\cal V_k$ are saturated subsheaves of $\cal O( E_k)$, they exits a {\em big} open
subset $\cal U$ of $\PTX$ (an open subset $\cal U$ of $\PTX$ is big if  $\codim\,
\PTX\backslash\cal U\geq 2$) and subbundles $V_k$ of $ E_k$ defined on $\cal U$ such that the
restriction of the $\cV_k$'s to $\cal U$ are the sheaves of
sections of the $V_k$'s. On $\cal U$
we let $V$ be the subbundle $\oplus_{0\leq k\leq r} V_k$, so that $\cV_{|\cal U}={\cal
  O}_{\cal U}(V)$.

Observe that on the regular locus
$\cal R(\b)$ of $\b$, the rank of $\b^k$, as a vector bundle morphism from $E\otimes
L^k$ to $E$, is constant. Hence on this open subset the formulas used above to define the
subsheaves $\F_k$ of $\cE$ in fact define subbundles $F_k$ of $E$ such that ${\F_k}_{|{\cal R}(\b)}=
\cal O_{\cal R(\b)}(F_k)$. Therefore, on $\cal R(\b)$, the 
  subbundles $V_k$ such that ${\cV_k}_{|{\cal R}(\b)}=\cal O(V_k)$ are given by $V_k=E_k\cap F_{r-2k}$
and we may assume that $\cal R(\b)$ is contained and dense in $\cal U$.

\medskip

\begin{lemm}
  On the big open set $\cal U$, the subsheaf $\cV$ defines a reduction $P_{K\cap Q}$ of the
  structure group of $P_K$ to the subgroup $K\cap Q\subset K$.
\end{lemm}

\begin{proof}
We begin by working on $\cal R(\b)\subset\cal U$. We view an element $p$ of $P_K$ above
$\xi\in\PTX$ as an isomorphism between the fiber $ E_\xi$ of $ E=P_K(\E)$ and the model space $\E$.
The component $\b$ of the Higgs field is a section of $P_K(\fu_-)\otimes L^\vee\subset
P_K(\End(\E))\otimes L^\vee$.

Since for 
all $\xi\in\cal R(\b)$ and all $\eta\in L_\xi$, $\eta\neq 0$, we have that $\b_\xi(\eta)$ has rank
$r$, there exists $p\in (P_K)_\xi$ such that $p\circ\b_\xi(\eta)\circ p^{-1}=y\in\fu_-\subset\End (\E)$, so
that $p (V_\xi)=\V$. Therefore, on
$\cal R(\b)$, by Proposition~\ref{prop:stabV}~(\ref{item:stabilizer-V}),
the subbundle $V$ of $ E$ defines a (holomorphic) reduction $P_{K\cap Q}$ of the
structure group of $P_K$ to the subgroup $K\cap Q$ of $K$ ($Q$ is the normalizer in $G$ of the
parabolic subalgebra $\fq$ of Definition~\ref{defi:V}). Explicitely $P_{K\cap Q}=\{p\in P_K\,\mid\,
p( V_{\pi(p)})=\V \}$.

We now work on $\cal U$. Enlarge the structure group of $P_K$ to $\GL(\E)$. The subbundle
$V=\oplus_{k=0}^r V_k$ of $ E$ defines a reduction $P_S$ of the structure group of $P_{\GL(\E)}$ to the
stabilizer $S$ of $\V$ in $\GL(\E)$ by setting $P_S=\{p\in P_{\GL(\E)}\,\mid\, p (V_{\pi(p)})=\V
\}$.

Let $B\subset\cal U$ be an open ball on which $P_K$ is trivial. Then the reductions $P_{K\cap Q}$ of
$P_K$ on $B\cap\cal R(\b)$
and $P_S$ of $P_{\GL(\E)}$ on $B$ are respectively given by holomorphic maps $\sigma:B\cap\cal R(\b)\pfd
K/(K\cap Q)$ and $s:B\pfd \GL(\E)/S$. Moreover, if $\iota$ denotes the natural map $K/(K\cap Q)\pfd
\GL(\E)/S$, which is injective, then we have $s=\iota\circ\sigma$ on $B\cap\cal R(\b)$. Since
$K/(K\cap Q)$ is compact, its image by $\iota$ is closed in $\GL(\E)/S$. Therefore, since $B\cap\cal
R(\b)$ is dense in $B$, $s$ maps $B$ to $\iota(K/(K\cap Q))$. This means that the reduction
$P_{K\cap Q}$ initially defined on $\cal R(\b)$ extends to $\cal U$.
\end{proof}

We deduce that

\begin{prop}\label{prop:VHiggs}
The subsheaf $\cV$ is a leafwise Higgs subsheaf of the Higgs bundle $(E,\t)$ on $\PTX$.
\end{prop}

\begin{proof}
  By Proposition~\ref{prop:algHiggs}, we know that $y$ and $\fu_+$ stabilize
$\V$. Therefore, on $\cal R(\b)$, the two components $\b$ and $\g$ of the Higgs field
stabilize the subsheaf $\cV$ since it is the
sheaf of sections of the subbundle $V=P_{K\cap Q}(\V)$
of $ E=P_{K\cap Q}(\E)$. By continuity, this still holds on $\cal U$ since on this big open set
$\cV$ is also the sheaf of section of $V=P_{K\cap Q}(\V)$.
Now,  $\cV$ is a saturated, hence normal, subsheaf of
$\cal O( E)$ by definition. Hence the restriction map $\cV(\PTX)\pfd \cV(\cal U)$ is an isomorphism
since $\codim\,\PTX\backslash\cal U\geq 2$.  Therefore $\cV$ is indeed a leafwise Higgs subsheaf of
$( E,\t)$ on $\PTX$.
\end{proof}

\medskip

Instead of working with $\b$ on the  Higgs bundle $(E,\t)$, we can consider ${}^t\g$ on the
dual Higgs bundle $(E^\vee,{}^t\t)$ and exactly the same reasoning  yields a leafwise Higgs subsheaf
$\cV'$ of $(E^\vee,{}^t\t)$, see Remark~\ref{rem:dual}.

\subsection{Proof of the Milnor-Wood inequality}\label{subsec:proofMW}\hfill\medskip

Together with the computation of the slopes of the $H$-submodules $\V_k$ of
$\E$, the construction of the leafwise Higgs subsheaves $\cV$ of $(E,\t)$ and $\cV'$ of $(E^\vee,{}^t\t)$
gives the Milnor Wood inequality:

\begin{theo}\label{theo:MW} We have the inequalities
  $\deg_\T(E_0)+\frac{\rk\b}{2}\, \deg_\T(L) \leq\mu_\T(\cV)\leq 0$ and
  $\deg_\T(E_0^\vee)+\frac{\rk\g}{2}\, \deg_\T(L) \leq \mu_\T(\cV')\leq 0$, therefore
  $|\deg_\T(E_0)|\leq \frac p2\,\deg_\T(L^\vee)$.
\end{theo}

\begin{proof}
We begin with the inequality $\deg_\T(E_0)+\frac{\rk\b}{2}\,\deg_\T(L)\leq\mu_\T(\cV)$. Let $n_k:=\dim \V_k$.
First, recall that Proposition~\ref{prop:stabV}~(\ref{item:reductive-Q}) states that the unipotent
radical of $Q$ acts trivially on $\bV$.
Therefore so does the unipotent radical of $K \cap Q$. Thus, in fact, $\bV$ is a $(K \cap Q)/R_u(K
\cap Q)$-module, and
$(K \cap Q)/R_u(K \cap Q) \simeq H$ is reductive.
Thus we may apply Proposition~\ref{prop:calcul-pente} and deduce that the $K\cap Q$-representations
$(\det\, \V_{k+1})^{n_1n_k}$ and
$(\det\, \V_k)^{n_1n_{k+1}}\otimes(\det\,\V_1)^{n_kn_{k+1}}\otimes (\V_0^\star)^{n_1n_kn_{k+1}}$ are
isomorphic. On the big open set $\cal U\subset\PTX$, we have $\cal V_k=\cal O(V_k)$ where $V_k=P_{K\cap
  Q}(\V_k)$. Therefore on $\cal U$, and hence on $\PTX$, the line bundles $(\det\,
\cal V_{k+1})^{n_1n_k}$ and $(\det\, \cal V_k)^{n_1n_{k+1}}\otimes(\det\,\cal
V_1)^{n_kn_{k+1}}\otimes (\cal V_0^\star)^{n_1n_kn_{k+1}}$ are
isomorphic. This implies that $\mu_\T(\cal V_{k+1})=\mu_\T(\cal V_k)+\mu_\T(\cV_0^\vee\otimes\cal V_1)$,
i.e. that $\mu_\T(\cal V_k)=\deg_\T(\cal V_0)+k\,\mu_\T(\cV_0^\vee\otimes\cal V_1)$.

Let $r$ be short for $\rk\b$. Since $\b^r:\cal V_0\otimes L^r\pfd\cal V_r$ is not zero, we also have $\mu_\T(\cal V_r)\geq \mu_\T(\cal
V_0)+r\,\mu_\T(L)$ so that $\mu_\T(\cV_0^\vee\otimes\cal V_1)\geq \deg_\T (L)$.

Finally, remembering that $n_k=n_{r-k}$ by Lemma~\ref{lemm:dim-Vi} and that $\cV_0=E_0$, we get
$$
\begin{array}{rcl}
  2\,\deg_\T(\cal V) & = & \ds\sum_{k=0}^r \deg_\T(\cal V_k)+\sum_{k=0}^r \deg_\T(\cal V_{r-k}) \\
& = & \ds\sum_{k=0}^r \left({ n_k\,\mu_\T(\cal V_k)+n_{r-k}\, \mu_\T(\cal V_{r-k})}\right) \\
& \geq &  \ds\sum_{k=0}^r n_k \left(\deg_\T(\cal V_0)+k\,\deg_\T(L)+ \deg_\T(\cal V_0)+(r-k)\,\deg_\T(L)\right) \\
& \geq & (\dim\V)\,(2\,\deg_\T(E_0)+r\,\deg_\T(L))
\end{array}
$$

The inequality $\mu_\T(\cV)\leq 0$ follows from Propositions~\ref{prop:VHiggs} and~\ref{prop:polystab}.

Finally the inequalities involving $E_0^\vee$, $\g$ and $\cV'$ are proved exactly in the same way. The
conclusion follows since $\rk\b\leq p$ and $\rk\g\leq p$.
\end{proof}

\medskip

In case of equality in Theorem~\ref{theo:MW}, we have:

\begin{prop}\label{prop:equal}
  Assume that $\deg_\T(E_0)+ \frac{\rk\b}{2}\, \deg_\T(L) =0$. Then
\begin{enumerate}
\item on the regular locus $\cal R(\b)=\PTX\backslash\S(\b)$ of $\b$, the orthogonal complement
  $V^\perp=\oplus V_i^\perp$ of the subbundle $V=\oplus V_i$ of $E$ w.r.t. the harmonic metric is
  stable under the Higgs field $\t:E\otimes L\pfd E$;
\item the regular locus $\cal R(\ovb)\subset X$ of $\ovb$ is (open and) dense in $X$.
\end{enumerate}
Similarly, if $\deg_\T(E_0^\vee)+\frac{\rk\g}{2}\, \deg_\T(L)=0$, then the orthogonal complement
  of $V' \subset E^\vee$ is stable by ${}^t\t:E^\vee\otimes L\pfd E^\vee$ on the regular locus ${\cal
    R}(\g)\subset\PTX$ of $\g$ and ${\cal R}(\ovg)\subset X$ is (open and) dense in $X$.
\end{prop}

\begin{proof}
  The first point follows from the discussion after the definition of the subsheaves $\cV_k$ and the
  polystability property (2) in Proposition~\ref{prop:polystab}, since our hypothesis implies that
  $\deg_\T \cV=0$ by Proposition~\ref{theo:MW}.
Proposition~\ref{prop:polystab}~(2) implies that the singular locus $\S(\b)\subset\PTX$ is saturated
under the tautological foliation $\T$, see the proof of
\cite{KM}*{Lemma~4.5}. This, together with M. Ratner's results on unipotent flows, in turn implies
the second point of the proposition, see
\cite{KM}*{Proposition~3.6}.
\end{proof}

\section{Maximal representations}\label{sec:max}

Maximal representations $\rho:\G\pfd G_\R$, where $\G$ is a uniform lattice in $\SU(1,n)$ with
$n\geq 2$ and $G_\R$ is a {\em classical} Lie group of Hermitian type,
were classified in \cite{KM}. Therefore we focus here on {\em exceptional} targets, namely $G_\R$ is either
${\rm E}_{6(-14)}$, which is not of tube type, or ${\rm E}_{7(-25)}$, which is.

In Section~\ref{subsec:tube} we exclude the possibility of maximal representations in ${\rm E}_{7(-25)}$.
In fact, our uniform approach allows to easily prove that maximal
representations in tube type target groups $G_\R$ do not exist.  The case of ${\rm E}_{6(-14)}$ is
treated in Section~\ref{subsec:E6}.

\subsection{Tube type targets}\label{subsec:tube}\hfill

\medskip

We prove that whenever $G_\R$ has tube type and $n\geq 2$, representations from $\G$ to
$G_\R$ satisfy an inequality  stronger than the Milnor-Wood inequality, preventing any
representation in such a group to be maximal:

\begin{prop}
 \label{prop:tube}
 Let $\G$ be a uniform lattice in $\SU(1,n)$, with $n\geq 2$, and let $X=\G\backslash\H$.
 Assume that $G_\R$ has tube type and let $p$ be the real rank of $G_\R$. Let $\rho$ be a
 representation $\G\pfd G_\R$. Then
 $$
 |\tau(\rho)|\,\leq\, \max\left({p-1,\frac p2\cdot \frac{n+1}{n}}\right)\, \vol(X)\,<\,p\, \vol(X).
 $$
\end{prop}

\begin{proof}
We may assume that $\tau(\rho)>0$. Recall that we assumed without loss of generality that $\rho$
is reductive (Assumption \ref{assu:reductive}). Then, the constructions of
\textsection \ref{sec:HiggsFoliated} and \textsection \ref{sec:MW}
are valid and the inequality of the Proposition is equivalent to the inequality
$$
\deg_\T(E_0) \,\leq\,
\max\left({\frac{p-1}{2},\frac {p}{2} \cdot \frac{n+1}{2n}}\right)\, \deg_\T(L^\vee)\,
\, < \, \frac p2 \, \deg_\T(L^\vee)\, .
$$
We use freely the notation of \textsection \ref{sec:MW}.  If the generic rank of $\b$ on the
projectivized tangent bundle $\PTX$
of $X$ is $\leq p-1$ then we are done by Theorem~\ref{theo:MW}. Therefore we may assume that the
generic rank of $\b$ on $\PTX$ is $p$.

We come back to the Higgs bundle $(\ovE,\ovt)$ on $X$ and we consider $\ovb:\ovE\otimes T_X\pfd
\ovE$. The fact that $\rk\b=p$ implies that $\ovb^p$, seen as a morphism from
$\ovE_0\otimes\ovE_p^\vee$ to the $p$-th symmetric power $S^p\Omega^1_X$ of $\Omega^1_X$, is not
zero. Since $\Omega^1_X$ is a semistable bundle
over $X$ ($X$ is Kähler-Einstein), so is $S^p\Omega^1_X$. On the other hand,
$\ovE_0\otimes\ovE_p^\vee$ is also semistable
because it is a line bundle by Remark \ref{rema:tube-type}.
Therefore $\mu(\ovE_0\otimes\ovE_p^\vee)\leq \mu(S^p\Omega^1_X)$,
so that $\deg \ovE_0-\deg \ovE_p\leq p\,\mu(\Omega^1_X)$. Now, as explained in
Remark \ref{rema:tube-type},
the $K$-modules
$\E_p$ and $\E_0^\vee$ are isomorphic, so that $\deg\ovE_p=-\deg\ovE_0$.
Hence the
result, since by equations (\ref{equa:deg-canonique}) and (\ref{equa:deg-L}),
$\deg(\Omega^1_X)=\frac{n+1}{2}\, \deg_\T(L^\vee)$.
\end{proof}

\subsection{Target group ${\rm E}_{6(-14)}$}\label{subsec:E6}\hfill

\medskip

\subsubsection{Algebraic preliminaries}
In the case $G_\R={\rm E}_{6(-14)}$, the minuscule
representation of $G={\rm E}_6$ is the standard
representation of ${\rm E_6}$ on the 27 dimensional complex exceptional Jordan algebra $\E={\mathfrak
  J}_\C$. The real rank of ${\rm E}_{6(-14)}$ is 2 and $\E=\E_0\oplus\E_1\oplus\E_2$ with $\E_0$, $\E_1$,
and $\E_2$ of dimension 1, 16 and 10 respectively.

We start with a description of the $\Spin_{10}$-representations $\bE_1$ and $\bE_2$ in terms of octonions.
More precisely, by Proposition \ref{prop:Ei}(e), there is a $\Spin_{10}$-equivariant isomorphism
$\bE_1 \simeq \bE_0 \otimes \fu_-$.
Choosing a non-zero vector in $\bE_0$, this yields an isomorphism $\alpha : \bE_1 \to \fu_-$.
We consider the quadratic map $\kappa : \bE_1 \to \bE_2$ defined by
$\kappa(x)=\alpha(x) \cdot x$. This is a $\Spin_{10}$-equivariant quadratic map $\bE_1 \to \bE_2$.
As the next proposition shows, there is, up to a scale, only one such map.
It is certainly well-known to specialists, however we could not find an adequate reference:

\begin{prop}  \label{prop:spin-10}
 There is an identification of $\E_1$ with $\bO \oplus \bO$ and $\E_2$ with $\bC \oplus \bO \oplus \bC$
 such that $\kappa(u,v)=(N(u),uv,N(v))$.
\end{prop}
\begin{proof}
We consider the $\Spin_{10}$ half-spin representation $\E_1^*$. According to \cite{chevalley}, when we
restrict to $\Spin_8$, this representation splits as $\S^+ \oplus \S^-$, where $\S^\pm$ denote the two
half-spin representations of $\Spin_8$. Similarly, the $\Spin_{10}$ vector representation $\E_2$ splits as
$\bC \oplus \bV \oplus \bC$, where $\bV$ denotes the $8$-dimensional vector representation.

Now, the quadratic map $\kappa$ is given by a $\Spin_{10}$-equivariant injection
$\E_2 \subset \E_1^* \otimes \E_1^*$. Since there are equivariant maps $\S^+ \oplus \S^- \to \bV$,
$\S^+ \otimes \S^+ \to \bC$ and $\S^- \otimes \S^- \to \bV$, and no equivariant maps from other
factors in this tensor product to $\E_2$, $\kappa$ is of the form
$\kappa(s_+,s_-) = ( \psi_+(s_+) , \varphi(s_+,s_-) , \psi_-(s_-) )$, for some equivariant maps
$\psi_+,\varphi$ and $\psi_-$. None of these maps can vanish, otherwise the image of $\kappa$
would be degenerate. Moreover, by triality, there are, up to scale, only one such map, which can be given,
once $\S^+,\S^-$ and $\bV$ are identified with the space of octonions $\bO$, by the formulas:
$\psi_+(s_+)=N(s_+),\varphi(s_+,s_-)=s_+s_-$ and $\psi_-(s_-)=N(s_-)$. The proposition follows.
\end{proof}

Given $x \in \fu_-$ resp. $y \in \fu_+$, $x$ resp. $y$ defines linear maps $\E_0 \to \E_1$ and $\E_1 \to \E_2$ resp.
$\E_1 \to \E_0$ and $\E_2 \to \E_1$. We denote these maps by $\lambda_1(x),\lambda_2(x)$
resp. $\mu_1(y),\mu_2(y)$. We thus have maps
$\lambda_1(x):\E_0 \to \E_1\ , \ \lambda_2(x):\E_1 \to \E_2$ and
$\mu_1(y):\E_1 \to \E_0 \ ,\ \mu_2(y):\E_2 \to \E_1$.

We can deduce from the explicit formula above some information about maps $\lambda_2(x)$:

\begin{prop}   \label{prop:fu-}
 Let $x,y \in \E_1 \simeq \fu_-$.
 Assume $x \neq 0$ and $y \neq 0$.
 \begin{enumerate}[(a)]
  \item $x$ has rank one \iff $\kappa(x) = 0$.
  \item $x$ has rank one \iff $\lambda_2(x)$ has rank $5$.
  \item $x$ has rank two \iff $\lambda_2(x)$ has rank $9$.
  \item Assume that any non trivial linear combination of $x$ and $y$ has rank $2$. Then
  $\dim ( \Ker \lambda_2(x) \cap \Ker \lambda_2(y) ) \leq 3$.
  \item Assume that $x$ and $y$ have rank $1$ and $\dim ( \im \lambda_2(x) \cap \im \lambda_2(y) ) \geq 4$.
  Then $x$ and $y$ are proportional.
 \end{enumerate}
\end{prop}
\begin{proof}
 We use the above isomorphism $\E_1 \simeq \bO \oplus \bO$. According to \cite{igusa}, there are exactly
 $3$ orbits in $\E_1$ under $\Spin_{10} \times \C^*$. Let $u \in \bO$ such that $N(u)=0$. We have
 $\kappa(u,0)=(0,0,0)$ and $\kappa(1,0)=(1,0,0)$. Thus, $(u,0)$ and $(1,0)$ cannot be in the same orbit.
 It follows that $(u,0)$ has rank $1$ and $(1,0)$ has rank $2$ and statement $(a)$ of the proposition is
 proved.

 Let $\tilde \kappa : \E_1 \times \E_1 \to \E_2$ be the polarization of $\kappa$, namely, the unique
 symmetric bilinear map such that $\tilde \kappa ( x , x ) = \kappa (x)$ for all $x$ in $\E_1$.
 We have $\lambda_2(x)=\tilde \kappa ( x , \cdot )$. Thus the image of $\lambda_2(u,0)$ is the set of triples
 $(t,z,t')$ with $t \in \bC$ arbitrary, $z$ a right multiple of $u$ in $\bO$, and $t'=0$: this space
 has dimension $5$. On the other hand, the image of $\lambda_2(1,0)$ is the set of triples
 $(t,z,t')$ with $t$ and $z$ arbitrary and $t'=0$. It has dimension $9$. Points $(b)$ and $(c)$ are
 proved.

 For point $(d)$, let us assume that any non trivial linear combination of $x$ and $y$ has rank $2$. Thanks to the
 result of Igusa, we may assume that $x=(1,0)$. Writing $y=(a,b)$, the assumption implies that $b \neq 0$
 (in fact, if $y=(a,0)$, then some linear combination of $x$ and $y$ will be of the form $(u,0)$ with
 $N(u)=0$). The kernel of $\lambda_2(x)$ is the space of elements of the form $(u,0)$ with $\scal{u,1}=0$.
 If such an element is in the kernel of $\lambda_2(x)$ then $bu=0$ and so $N(u)=0$. Thus, the intersection of
 the kernels of $\lambda_2(x)$ and $\lambda_2(y)$ is isomorphic to an isotropic subspace of the space of
 octonions $u$ with $\scal{u,1}=0$. Such an isotropic subspace can have at most dimension $3$.

 Finally, let us assume that $x$ and $y$ have rank $1$ and that $\dim ( \im \lambda_2(x) \cap \im \lambda_2(y) ) \geq 4$.
 Then we may assume that $x=(u,0)$ with $N(u)=0$ as
 above. The image of $\lambda_2(x)$ is then the set of triples $(t,z,0)$ with $t$ arbitrary and $z$ of the
 form $uw$ for some octonion $w$. Thus, this space is an isotropic subspace of $\E_2$ of maximal dimension $5$.
 Using the $\Spin_{10}$-action, it follows that for any $x \in \E_1$ of rank $1$, the image of $\lambda_2(x)$
 is an isotropic subspace of dimension $5$. Since two maximal isotropic subspaces
 in the same family can intersect only
 in odd dimension, it follows from the hypothesis on $x$ and $y$
 that the images of $\lambda_2(x)$ and $\lambda_2(y)$ are equal. One can check
 that this implies that $y$ is proportional to $(u,0)=x$.
\end{proof}

We constructed a quadratic $\Spin_{10}$-equivariant map $\kappa : \E_1 \to \E_2$ identifying $\E_1$ with $\fu_-$ and using
the linear map $\E_1 \to \E_2$ given by $x \in \fu_-$. Similarly,
let $\iota : \E^*_1 \to \E_2^*$ be the quadratic equivariant map obtained identifying $\E_1^*$ with $\fu_+$ and using
the linear map ${}^t\mu_2(w):E_1^* \to \E_2^*$ given by $w \in \fu_+$.
With the same proof, we get information about $\fu_+$ and the linear maps $\mu_2(w)$:

\begin{prop}   \label{prop:fu+}
 Let $w,z \in \E_1^* \simeq \fu_+$.
 Assume $w \neq 0$ and $z \neq 0$.
 \begin{enumerate}[(a)]
  \item $w$ has rank one \iff $\iota(w) = 0$.
  \item $w$ has rank one \iff $\mu_2(w)$ has rank $5$.
  \item $w$ has rank two \iff $\mu_2(w)$ has rank $9$.
  \item Assume that any non trivial linear combination of $w$ and $z$ has rank $2$. Then
  $\dim ( \im \mu_2(w) \cap \im \mu_2(z) ) \leq 3$.
  \item Assume that $w$ and $z$ have rank $1$ and $\dim ( \Ker \mu_2(w) \cap \Ker \mu_2(z) ) \geq 4$.
  Then $w$ and $z$ are proportional.
 \end{enumerate}
\end{prop}

\medskip

\subsubsection{Maximal representations}

Let $\rho:\G\pfd {\rm E}_{6(-14)}$ be a reductive representation.
We may therefore
consider the  the Higgs bundle $(\ovE,\ovt)$ on $X$ and its pull-back $(E,\t)$ on $\PTX$
associated to  $\r$  and the representation of $\rm E_{6}$ on
$\E=\E_0\oplus\E_1\oplus\E_2$ as in Section~\ref{sec:MW}.

Recall that the components of the Higgs field $\ovt$ are
$$
P_K(\fu_-)\ni\ovb=:(\ovb_1,\ovb_2)\in\Big(\Hom(\ovE_0,\ovE_1)\otimes\Omega^1_X\Big)\oplus\Big(\Hom(\ovE_1,\ovE_2)\otimes\Omega^1_X\Big)
$$
and
$$
P_K(\fu_+)\ni\ovg=:(\ovg_1,\ovg_2)\in\Big(\Hom(\ovE_1,\ovE_0)\otimes\Omega^1_X\Big)\oplus\Big(\Hom(\ovE_2,\ovE_1)\otimes\Omega^1_X\Big)
$$

To lighten the notation, the fibers of the bundles $\ovE$, $\ovE_0$, $\ovE_1$ and
$\ovE_2$ above some $x\in X$ will also be denoted by $\ovE$, $\ovE_0$, $\ovE_1$ and
$\ovE_2$.

Propositions \ref{prop:fu-} and \ref{prop:fu+} immediately imply  the following:

\begin{lemm}\label{lem:ranks} Let $x\in X$ and $\xi$ be a holomorphic tangent vector at $x$.

  As an element of $\Hom(\ovE_1,\ovE_2)$, $\ovb_2(\xi)$ has rank
  $0$, $5$ or $9$. Moreover:
\begin{enumerate}[(i)]
\item If $\ovb(\xi)$ has rank 1, i.e. $\ovb_1(\xi)\neq 0$ but $\ovb_2(\xi)\ovb_1(\xi)=0$, then $\ovb_2(\xi):\ovE_1\pfd \ovE_2$ has rank $5$;

\item If $\ovb(\xi)$ has rank $2$, i.e. if $\ovb_2(\xi)\ovb_1(\xi)\neq 0$, then $\ovb_2(\xi):\ovE_1\pfd \ovE_2$ has rank $9$;

\item If any non trivial linear combination of $\ovb(\xi)$ and $\ovb(\eta)$ has rank
  $2$, then we have $\dim(\Ker\ovb_2(\xi)\cap\Ker\ovb_2(\eta))\leq 3$.
\end{enumerate}

Similarly, as an element of $\Hom(\ovE_2,\ovE_1)$, $\ovg_2(\xi)$ has rank
  $0$, $5$ or $9$. Moreover:
\begin{enumerate}[(a)]
\item If $\ovg(\xi)$ has rank $1$, i.e. $\ovg_2(\xi)\neq 0$ but $\ovg_1(\xi)\ovg_2(\xi)=0$, then $\ovg_2(\xi):\ovE_2\pfd \ovE_1$ has rank $5$;
\item If $\ovg(\xi)$ has rank $2$, i.e. if $\ovg_1(\xi)\ovg_2(\xi)\neq 0$, then $\ovg_2(\xi):\ovE_2\pfd \ovE_1$ has rank $9$;
\item If $\ovg(\xi)$ and $\ovg(\eta)$ have rank $1$ and $\dim (\Ker \ovg_2(\xi) \cap \Ker \ovg_2(\eta)) \geq 4$, then
$\ovg(\xi)$ and $\ovg(\eta)$ are colinear.
\end{enumerate}
\end{lemm}

\medskip

Thanks to this lemma, in case of equality in the Milnor-Wood inequality, we may prove

 \begin{prop}\label{prop:vanishing}
  If $\deg_\T(E_0)=\deg_\T(L^\vee)$ and $x\in\cal R(\ovb)$, then for all $\xi\in T_{X,x}$, $\ovg(\xi)=0$.

If $\deg_\T(E_0)=-\deg_\T(L^\vee)$ and $x\in\cal R(\ovg)$, then for all $\xi\in T_{X,x}$, $\ovb(\xi)=0$.
\end{prop}

\begin{proof}
We prove only the first assertion, the proof of the second one follows exactly the same lines. The
letters $\xi$ and $\eta$ will denote (holomorphic) tangent
vectors at $x$.

The equality $\deg_\T(E_0)=\deg_\T(L^\vee)$  and Theorem~\ref{theo:MW} imply that the generic rank
of $\b$ on $\PTX$ is 2. Therefore, since
$x$ belongs to the regular locus $\cal R(\ovb)$ of $\ovb$, for all $\xi\neq 0$ in $T_{X,x}$, the
rank of $\ovb(\xi)$ is 2, so that the rank of $\ovb_2(\xi)$ is 9 by Lemma~\ref{lem:ranks}.

\smallskip

We will make a crucial use of the
integrability relation $[\ovt,\ovt]=0$ of the Higgs field $\ovt$. This relation is equivalent to the
following three conditions:
$$
\left\{
\begin{array}{ll}
\ovg_1(\xi)\ovb_1(\eta) = \ovg_1(\eta)\ovb_1(\xi) & \mbox{ in }  \End(\ovE_0)\\
\ovb_1(\xi)\ovg_1(\eta)+\ovg_2(\xi)\ovb_2(\eta) =  \ovb_1(\eta)\ovg_1(\xi)+\ovg_2(\eta)\ovb_2(\xi) & \mbox{ in }  \End(\ovE_1) \\
\ovb_2(\xi)\ovg_2(\eta) = \ovb_2(\eta)\ovg_2(\xi) & \mbox{ in } \End(\ovE_2)
\end{array}
\right.
$$
which hold for all $\xi,\eta$.

\smallskip

Suppose first that there exists $\xi$ such that $\ovg_2(\xi):\ovE_2\pfd \ovE_1$ has rank 9. Consider the subspace
$W:=\Ker\ovg_1(\xi)\cap\Ker \ovb_2(\xi)\subset \ovE_1$. Since $\dim \ovE_1=16$, $\dim \Ker\ovg_1(\xi)=15$ and
$\dim\Ker\ovb_2(\xi)=7$, we have $\dim W\geq 6$. On this subspace, the second integrability condition
reads $\ovb_1(\xi)\ovg_1(\eta)+\ovg_2(\xi)\ovb_2(\eta) =0$ for all $\eta$. Therefore $\ovg_2(\xi)\ovb_2(\eta)(W)\subset \ovE_1$ is
1-dimensional. Because of our assumption on the rank of $\ovg_2(\xi)$, $\ovb_2(\eta)(W)$ is of
dimension at most 2, and this implies that $\dim W\cap\Ker \ovb_2(\eta)\geq 4$, hence that $\dim
\Ker\ovb_2(\xi)\cap\Ker\ovb_2(\eta)\geq 4$. We get a contradiction with Lemma~\ref{lem:ranks}$(iii)$.

\smallskip

Suppose now that for all $\xi\neq 0$, $\ovg_2(\xi)$ has rank 5. Fix $\xi\neq 0$, and let
$[\xi]$ be the class of $\xi$ in the fiber of $\PTX$ above $x$. Let  $V(\xi)=V_0(\xi)\oplus
V_1(\xi)\oplus V_2(\xi)$ be the fiber above $[\xi]$ of the subbundle $V$ of the Higgs bundle
$(E,\t)$ on $\PTX$. We have
$$
\left\{
\begin{array}{l}
V_0(\xi)=E_0=\ovE_0 \\
V_1(\xi)=E_1\cap F_0=E_1\cap(\Ker\b_{[\xi]}\oplus\Ker\b_{[\xi]}^2\cap\Im\b_{[\xi]})=\Ker\ovb_2(\xi)\oplus\Im\ovb_1(\xi)\\
V_2(\xi)=E_2\cap F_{-2}= E_2\cap(\Ker\b_{[\xi]}\cap\Im\b^2_{[\xi]})=\Im\ovb_2(\xi)\ovb_1(\xi)
\end{array}
\right.
$$
and we know by Proposition~\ref{prop:equal}~(1) that the orthogonal complement $V_1(\xi)^\perp\oplus
V_2(\xi)^\perp$ of $V_0(\xi)\oplus V_1(\xi)\oplus V_2(\xi)$ is invariant by $\ovt(\xi)$, in particular that
$\ovg_2(\xi)$ maps $V_2(\xi)^\perp$ to $V_1(\xi)^\perp$.

By the third integrability condition, $\ovg_2(\xi)$ maps
$\Ker\ovg_2(\eta)$ in $\Ker\ovb_2(\eta)$. Hence $\ovg_2(\xi)$ maps
$V_2(\xi)^\perp\cap\Ker\ovg_2(\eta)$ to
$\Ker\ovb_2(\eta)\cap V_1(\xi)^\perp$.

But $\ovb_2(\xi)$ is injective on $V_1(\xi)^\perp$ because $\Ker\ovb_2(\xi)\subset
V_1(\xi)$. Hence for $\eta$ close to $\xi$, $\ovb_2(\eta)$ is also injective on $V_1(\xi)^\perp$, so that
$V_2(\xi)^\perp\cap\Ker\ovg_2(\eta)\subset\Ker\ovg_2(\xi)$. Now, $\dim V_2(\xi)^\perp=9$ and
$\rk\ovg_2(\eta)=5$, thus   $V_2(\xi)^\perp\cap\Ker\ovg_2(\eta)$ is at least 4
dimensional, and so is $\Ker\ovg_2(\xi)\cap\Ker\ovg_2(\eta)$. Again, this implies
by Lemma~\ref{lem:ranks}$(c)$ that $\ovg_2(\xi)$ and
$\ovg_2(\eta)$ are colinear, a contradiction since we assumed that all the $\ovg_2(\zeta)$,
$\zeta\neq 0$, have rank 5.

\smallskip

We conclude that there exists $\xi\neq 0$ such that $\ovg_2(\xi)=0$. Then also $\ovg_1(\xi)=0$ and by
the second integrability condition, for all $\eta$, $\ovb_1(\xi)\ovg_1(\eta) =
\ovg_2(\eta)\ovb_2(\xi)$. Therefore
$\ovg_2(\eta)$ has rank at most $1$ on $\Im\ovb_2(\xi)$ which is 9 dimensional in $\ovE_2$, so that
$\ovg_2(\eta)$ has
rank at most 2, hence vanishes. Therefore $\ovg_2=0$ and $\ovg=0$ identically on $T_{X,x}$.
\end{proof}

\begin{theo}\label{thm:E6}
Let $\G$ be a uniform lattice in $\SU(1,n)$ with $n\geq 2$ and $\rho$ be a maximal representation of
$\G$ in ${\rm E}_{6(-14)}$. Then $n=2$ and there exists a holomorphic or antiholomorphic $\rho$-equivariant
embedding from $\mathbb H^2_\C$ to the symmetric space $M$
associated to ${\rm E}_{6(-14)}$.
\end{theo}

\begin{proof}
  By \cite{BIW}, maximal representations are reductive, and we may apply our previous results. We
  assume $\tau(\rho)>0$, the case $\tau(\rho)<0$ being handled similarly. By
  Proposition~\ref{prop:vanishing}, $\ovg$ vanishes on the regular locus ${\cal R}(\ovb)$ of $\ovb$. By
  Proposition~\ref{prop:equal}~(2), ${\cal R}(\ovb)$ is dense in $X$, so that $\ovg$ vanishes
  identically on $X$. This means that the $\rho$-equivariant harmonic map $f:\H\pfd M$ used to
  define the Higgs bundle $(\ovE,\ovt)$ is holomorphic.
\end{proof}

\subsection{Proof of the main results}\label{subsec:proof}\hfill

\medskip

In this subsection, we give detailed proofs of the theorem and corollaries given in the introduction,
although some of the arguments might be well-known to specialists.

\lpara 
We assume $\tau(\rho)>0$, the other case being similar. We may assume that $\rho$ is reductive by
\cite{BIW}.
Let then $f : \H \to M$ be a harmonic $\rho$-equivariant map
(such a map exists by \cite{CorletteFlatGBundles}).
By Theorem \ref{thm:E6}, Proposition~\ref{prop:tube} and \cite{KM}, $f$ is holomorphic.
By the Ahlfors-Schwarz lemma (cf. \cite{Royden}), since the holomorphic sectional
curvature is $-1$ on $\H$ and bounded above by $-\frac 1p$ on $M$,  we have the pointwise
inequality $f^\star \o_M\leq p \,\o$. The maximality of $\rho$ then implies that
$f^\star\o_M=p\,\o$.
Since there is equality in the Ahlfors-Schwarz lemma, $f$ is totally geodesic (see e.g. \cite{Royden}).

These properties imply that $f$ is a
so-called {\em tight} holomorphic totally geodesic map $\H\pfd M$ (as defined in \cite{Hamlet}).
\indication{In fact, $f$ is a local isometry, but since through any two points in $M$ there is a unique
geodesic, $f$ is a global isometry.}
Tight holomorphic maps between Hermitian
symmetric spaces were classified in~\cite{Hamlet}. If the symmetric space $M$ is not
irreducible, the map $f$ is tight if and only if all the induced maps to the irreducible factors of
$M$ are tight. We may therefore assume that $M$ is irreducible or equivalently that $G_\R$ is
simple. In this case, and since $n\geq 2$, tight holomorphic totally geodesic maps $\H\pfd M$ only exist when
$G_\R=\SU(p,q)$ with $q\geq pn$ or when $G_\R={\rm E}_{6(-14)}$ if $n=2$. They are deduced one from
another by composition by an element of $G_\R$.

\begin{rema}
  There is a small inaccuracy in~\cite{Hamlet}, where it is said that there are
  two ``tight regular'' (in the terminology of this paper) maximal subalgebras of
  ${\fe}_{6(-14)}$. In fact, only $\su(4,2)\subset {\fe}_{6(-14)}$ is a maximal subalgebra for
  these properties. This was confirmed to us by the
  author.
\end{rema}

This proves all the assertions of Theorem A and Corollary B except the uniqueness of the harmonic
map $\H\pfd M$ that is
$\rho$-equivariant. To prove
it, we need to have a closer look at $f$. It follows from \cite{Hamlet} (see also
\cite{KMrank2}*{Proposition 3.2})
that $f$ is equivariant with respect to a morphism of
Lie groups $\varphi : \SU(1,n) \to G_\R$ and that up to conjugacy of $\rho$, we may assume that $f$
and $\varphi$ are as follows:
\begin{itemize}
\item for $G_\R=\SU(p,q)$ with $q\geq pn$, $\varphi$ is the composition
$$
\SU(1,n)\hookrightarrow
{\SU(1,n)\times\cdots\times\SU(1,n)}
\hookrightarrow\SU(p,pn)\hookrightarrow
\SU(p,q)\,;
$$

\item for $G_\R={\rm E}_{6(-14)}$ and $n=2$, $\varphi$ is the composition
$$
\SU(1,2) \hookrightarrow
\SU(1,2)\times\SU(1,2)\hookrightarrow\SU(2,4)\hookrightarrow {\rm E}_{6(-14)}.
$$
where the last morphism is detailed in the proof of Lemma \ref{lemm:Z} below.
\end{itemize}

In both cases, the image $N$ of $f$ in $M=G_\R/K_\R$ is the orbit of $o=K_\R$ under
$H_\R:=\varphi(\SU(1,n))\subset G_\R$.

We now describe the centralizer $Z_\R$ of $H_\R$ in $G_\R$.
In case $G_\R = \SU(p,q)$, let $GZ_\R$ denote the group $\U(p) \times \U(q-pn)$ and let
$\chi : GZ_\R \to \U(1)$ be the character defined by $\chi(x,y)=\det(x)^{n+1} \cdot \det(y)$.
In case $G_\R = {\rm E}_{6(-14)}$, let $GZ_\R = \U(2) \times \U(2)$ and let $\chi : GZ_\R \to \U(1)$
be the character defined by $\chi(x,y)=\det(x)^{21} \cdot \det(y)^6$.  Then:

\begin{lemm}
 \label{lemm:Z}
 The centraliser $Z_\R$ of $H_\R$ in $G_\R$ is a subgroup of $K_\R$ (hence it is compact). It is
 isomorphic to the kernel of $\chi$ in $GZ_\R$.
\end{lemm}
\begin{proof}
In the case of $\SU(p,q)$, the description of $\varphi$ given above shows that the standard
representation $\C^{p+q}$ of $\SU(p,q)$, when seen as a representation of $\SU(1,n)$ via
$\varphi$, splits as
$$
\C^{p+q}=\C^{p+pn}\oplus\C^r=\C^{1+n}\otimes\C^p\oplus\C^r,
$$
where $\C^{1+n}$ is the standard representation of $\SU(1,n)$ and $r=q-pn$.
To conclude, we argue as follows. Let $g \in Z_\R$. Then $g$ yields an endomorphism
of the $H_\R$-module $\C^{1+n}\otimes\C^p\oplus\C^r$. Since by Schur's lemma
such an endomorphism will
preserve isotypic factors, we see that
$g$ must preserve the factors $\C^{1+n}\otimes\C^p$ and $\C^r$. Moreover it is known that
it must act by an element of $U(p)$ on the first factor, so that it belongs to $GZ_\R$.

In the case of ${\rm E}_{6(-14)}$, we use a model of the $27$-dimensional representation $\E$ given
by Manivel in \cite{manivel}*{Example 3 p.464}.
There is a subgroup in ${\rm E}_{6(-14)}$ isomorphic to $\SU(2,4) \times \SU(2)$ and $\E$ splits as
$\wedge^2 \bU \oplus \bU \otimes \bA$, where
$\bU$ resp. $\bA$ have complex dimension $6$ resp. $2$.
Here we restrict further to $\SU(1,2) \times \SU(2)$, where the first factor $\SU(1,2)$ is
diagonally embedded in $\SU(2,4)$, meaning that the representation
$\bU$ splits as $\bV \otimes \bB$, with $\dim \bV = 3$ and $\dim \bB=2$.
We get
$$
\E \simeq \wedge^2(\bV \otimes \bB) \oplus \bV \otimes \bA \otimes \bB
\simeq \wedge^2 \bV \otimes S^2 \bB \oplus S^2 \bV \otimes \wedge^2 \bB \oplus \bV \otimes \bA \otimes \bB.
$$
As in the case of $\SU(p,q)$, an element $g$ in the centralizer of $H_\Z$ will yield a
$H_\Z$-equivariant endomorphism, and will preserve each of these factors.
Since it is an element of the group $E_6$, one sees that it must be given by an
element in $\U(\bA) \times \U(\bB)$.

The computation of the character $\chi$ is done as follows. If $f=(x,y) \in \U(\bA) \times \U(\bB)$,
then the determinant of the action of $f$ on
$\E$ is the product of the determinants of the actions of $f$ on $\wedge^2(\bV \otimes \bB)$ and on
$\bV \otimes \bA \otimes \bB$.
The action on $\bV \otimes \bA \otimes \bB$ has determinant $\det(x)^6\det(y)^6$, and the action on
$\wedge^2(\bV \otimes \bB)$
has determinant $\det(y)^{15}$.
\end{proof}

\begin{rema}
 \label{rema:compactness}
 The compactness of $Z_\R$ is proved in greater generality in \cite{BIW}*{Theorem 3}.
\end{rema}

\begin{lemm} \label{lemm:fixator}
 The fixator of $N=f(\H)$ in $G_\R$ is exactly $Z_\R$. The stabilizer of $N$ in $G_\R$ is the
 almost direct product $H_\R\cdot Z_\R$.
\end{lemm}
\begin{proof}
 Let $o=K_\R \in N$ be the base point of $M$.
 Let us denote by ${\rm Fix}(N) \subset G_\R$ the subgroup of elements which fix all the elements in $N$.
 We want to prove that ${\rm Fix}(N)=Z_\R$. We have an inclusion $Z_\R \subset {\rm Fix}(N)$. Indeed, if $h \in H_\R$ and
 $z \in Z_\R$, then $z\cdot o = o$ since $Z_\R\subset K_\R$. Thus, since $g$ and $h$ commute, $z
 \cdot ( h \cdot o ) = h \cdot ( z \cdot o ) = h \cdot o$.

 The subgroup $H_\R$ may be defined refering only to $N$ as follows.
 Let $\fg_\R = \fk_\R \oplus \fp_\R$ be the Cartan decomposition of $\fg_\R$. The tangent space
 $T_oN$ identifies with a subspace of $\fp_\R$ that we denote by $\fq_\R$.
 The space $\fq_\R$ defines a Lie triple system, so that $\fh_\R := [\fq_\R,\fq_\R] \oplus \fq_\R
 \subset \fg_\R$ is a Lie subalgebra. Then, $H_\R$ is the connected Lie group of $G_\R$
 with Lie algebra $\fh_\R$.

 For the reverse inclusion we need to prove that ${\rm Fix}(N) \subset Z_\R$.
 It follows from the given description of $H_\R$ that $H_\R$ is normalized by ${\rm Fix}(N)$.
 Let $g \in {\rm Fix}(N)$ and $h \in H_\R$. Then the commutator
 $ghg^{-1}h^{-1}$ belongs to $H_\R$ and acts trivially on $N$. Thus, it belongs to the center of
 $H_\R$. Since this center is finite, the connexity of
 $H_\R$ implies that $ghg^{-1}h^{-1}$ is the neutral element.

 Since the automorphism group of $N$ is $H_\R$, the second assertion of the Lemma follows from the first.
\end{proof}

\lpara

\noindent{\bf Proof of Corollary C:} The facts that $\rho$ is dicrete and faithful and that
  $\rho(\G)$ acts cocompactly on $N$  follow from the $\rho$-equivariance of the totally geodesic
  embedding $f$. The reductivity of $\rho$ has been already asserted and the compactness of $Z_\R$
  was established in Lemma \ref{lemm:Z}. Now,
given $\gamma \in \Gamma$, the equivariance of $f$ w.r.t. $\rho$ and $\varphi$ means that
$\rho(\gamma)$ and $\varphi(\gamma)$ have the same action on $N$.
We let $\rhoc(\gamma)=\rho(\gamma)\varphi(\gamma)^{-1}$. This is an element of the fixator of $N$,
which is equal to the centralizer $Z_\R$ of $H_\R$ by Lemma \ref{lemm:fixator}.  Since
$\varphi(\gamma) \in H_\R$ by definition of $\varphi$, the elements $\varphi(\gamma)$ and $\rhoc(\gamma)$ commute.
It follows that $\varphi(\gamma)$ and $\rho(\gamma)$ commute, and that $\rhoc$ is a morphism of groups.
\qed

\lpara

\noindent{\bf Proof of the uniqueness of $f$:}
by the uniqueness statement for tight holomorphic
totally geodesic maps
$\H\pfd M$, we know that if $f':\H\pfd M$ is
another $\rho$-equivariant harmonic map, then there exists $g\in G_\R$ such that $f'=g\circ f$. By
$\rho$-equivariance of $f$ and $f'$, we have that
$$
\rho(\g)\circ g(x)=g\circ\rho(\g)(x),\,\forall \g\in\G\mbox{ and }\forall x\in N.
$$
It follows that $g \cdot N$ is $\rho(\Gamma)$-stable.
Thus the map $d_{g\cdot N}:N\pfd\R$, $x\mapsto d(x,g\cdot N)$, where $d$ denotes the distance in
$M$, is invariant under the cocompact action of $\rho(\G)$ on $N$. It is therefore
bounded. Since it is moreover convex (\cite{BH}*{p. 178}), it is constant, equal to $a$, say. In the same way, the map $d_N
:g\cdot N\pfd\R$, $y\mapsto d(y,N)$ is also constant equal to $a$.

If $a>0$ it follows from the sandwich
lemma (\cite{BH}*{p. 182}) that the convex hull of $N\cup g\cdot N$ in $M$ is isometric to the product $N\times
[0,a]$. This implies that there exists a tangent vector $v\in T_oM\simeq \fp_\R$, orthogonal to
$T_oN\simeq \fq_\R$ such that $[v,u]=0$ for all $u\in\fq_\R$. Indeed the
norm (for the Killing form) of $[v,u]\in \fg_\R$ is up to a constant the sectional curvature of the
plane generated by the tangent vectors $u$ and $ v$, which is $0$ since they belong to different factors
of a Riemannian product.
In this case the 1-parameter group of transvections along
the geodesic defined by $v$ is included in the centralizer $Z_\R$ of $H_\R$, a contradiction since
$Z_\R$ is compact.

Hence $a=0$ and $g\cdot N=N$. Therefore there exist $h\in H_\R$ and $z\in Z_\R$ such
that $g=hz=zh$. The above commutation relation between $\rho(\g)=\varphi(\g)\rhoc(\g)$ and $g$ on $N$
means that $\rho(\g)g\rho(\g)^{-1}g^{-1}$ fixes $N$ pointwise and hence belongs to $Z_\R$ by
Lemma~\ref{lemm:fixator}. Hence for
all $\g\in\G$ we obtain that $\varphi(\g)h\varphi(\g)^{-1}h^{-1}$ belongs to $Z_\R\cap H_\R$
(recall that $\rhoc(\gamma) \in Z_\R$). Now $\G$
is Zariski dense in $\SU(1,n)$ by the Borel density theorem and we deduce that
$\varphi(x)h\varphi(x)^{-1}h^{-1}\in Z_\R\cap H_\R$ for all $x\in\SU(1,n)$.
Since $Z_\R\cap H_\R$ is finite and
$\SU(1,n)$ is connected, $h\in Z_\R$. Therefore $g\in Z_\R$ and $f'=f$.
\qed

\begin{rema}
 \label{rema:1-connected}
 If we drop the assumption that $G$ is simply-connected, then $\E$ might no longer be a representation of $G$ and our constructions
 cannot be made. However, in this case, letting $\tilde G$ be the simply connected cover of $G$ and $\E$ the cominuscule representation
 of $\tilde G$ that we have been considering, there is an integer $k$ such that $\E^{\otimes k}$ is a representation of $G$. The
 arguments given in the article can be adapted with the representation $\E^{\otimes k}$ instead of $\E$, and the main results
 (Theorem A, Corollary B and Corollary C) remain true without the simple-connectedness assumption.
\end{rema}

\indication{
\begin{proof}
 We denote by $\Zmax$ the maximum value $\scal{\Lambda,z}$ for a weight $\Lambda$ in $X(\Ek)$.
 Since such a weight is of the form $\chi_1 + \cdots + \chi_k$ for some weights $\chi_i \in X(\E)$, we have
 $\Zmax = k \zmax$. Similarly the highest value of $\scal{\Lambda,h}$ for $\Lambda \in X(\Ek)$ is $kr$.
 We define as in Definition \ref{defi:Ei} and \ref{defi:V} the subspaces $\Ek_i$ and
 $\V^k_i$ by the relations
 $ \displaystyle \Ek_i = \bigoplus_{\Lambda : \scal{\Lambda,z} = k\zmax-2i} \Ek_\Lambda $ and
 $ \displaystyle \V^k_i = \bigoplus_{\Lambda \in X(\Ek_i) : \scal{\Lambda,h}=kr-2i} \E_\chi \ .$
 We set $\displaystyle \V^k = \bigoplus_i \V^k_i$.
 If $\Lambda = \chi_1 + \cdots + \chi_k \in X(\Ek_i)$, then for all $l$ we have
 $\chi_l \in \E_{j_l}$ with integers $j_l$ such that $i=j_1 + \cdots + j_k$. We have the inequality
 $\scal{\chi_l,h} \leq r-2j_l$. Thus, if $\Ek_\Lambda \subset \V^k_i$, this implies
 that for each $l$, $\E_{\chi_l} \subset \V_l$. We deduce that
 $\V^k = \V^{\otimes k}$. It follows that at the level of sheaves, we will be able to
 define a subsheaf of $\cEk$ by the same trick as in (\ref{equa:filtration-y}).
 Since passing from $\E$ to $\Ek$ and from $\V$ to $\Vk$ just multiplies the slopes by $k$,
 we get the same Milnor-Wood inequality. In the case of equality, the arguments of
 Section \ref{subsec:E6} are still valid and we get that $\gamma=0$,
 proving that the map $f$ is holomorphic.
\end{proof}
}

\bibliographystyle{alpha}
\bibliography{milnor}

\begin{bibdiv}
\begin{biblist}

\bib{biquard}{article}{
      author={Biquard, Olivier},
      author={Garcia-Prada, Oscar},
      author={Rubio, Roberto},
       title={Higgs bundles, {T}oledo invariant and the {C}ayley
  correspondence},
        date={2015},
     journal={arXiv:1511.07751},
}

\bib{BH}{book}{
      author={Bridson, M.~R.},
      author={Haefliger, A.},
       title={Metric spaces of non-positive curvature},
      series={Grundlehren der Mathematischen Wissenschaften},
   publisher={Springer-Verlag, Berlin},
        date={1999},
      volume={319},
}

\bib{BI07}{article}{
      author={Burger, Marc},
      author={Iozzi, Alessandra},
       title={Bounded differential forms, generalized {M}ilnor-{W}ood
  inequality and an application to deformation rigidity},
        date={2007},
        ISSN={0046-5755},
     journal={Geom. Dedicata},
      volume={125},
       pages={1\ndash 23},
         url={http://dx.doi.org/10.1007/s10711-006-9108-6},
      review={\MR{2322535}},
}

\bib{BIW}{article}{
      author={Burger, Marc},
      author={Iozzi, Alessandra},
      author={Wienhard, Anna},
       title={Tight homomorphisms and {H}ermitian symmetric spaces},
        date={2009},
        ISSN={1016-443X},
     journal={Geom. Funct. Anal.},
      volume={19},
      number={3},
       pages={678\ndash 721},
         url={http://dx.doi.org/10.1007/s00039-009-0020-8},
      review={\MR{2563767}},
}

\bib{buch}{article}{
      author={Buch, Anders~S.},
      author={Mihalcea, Leonardo~C.},
       title={Curve neighborhoods of {S}chubert varieties},
        date={2015},
        ISSN={0022-040X},
     journal={J. Differential Geom.},
      volume={99},
      number={2},
       pages={255\ndash 283},
         url={http://projecteuclid.org/euclid.jdg/1421415563},
      review={\MR{3302040}},
}

\bib{borel}{book}{
      author={Borel, Armand},
       title={Linear algebraic groups},
      series={Notes taken by Hyman Bass},
   publisher={W. A. Benjamin, Inc., New York-Amsterdam},
        date={1969},
      review={\MR{0251042}},
}

\bib{bourbaki}{book}{
      author={Bourbaki, N.},
       title={\'{E}l\'ements de math\'ematique. {F}asc. {XXXIV}. {G}roupes et
  alg\`ebres de {L}ie. {C}hapitre {IV}: {G}roupes de {C}oxeter et syst\`emes de
  {T}its. {C}hapitre {V}: {G}roupes engendr\'es par des r\'eflexions.
  {C}hapitre {VI}: syst\`emes de racines},
      series={Actualit\'es Scientifiques et Industrielles, No. 1337},
   publisher={Hermann, Paris},
        date={1968},
      review={\MR{0240238}},
}

\bib{chevalley}{book}{
      author={Chevalley, Claude},
       title={The algebraic theory of spinors and {C}lifford algebras},
   publisher={Springer-Verlag, Berlin},
        date={1997},
        ISBN={3-540-57063-2},
        note={Collected works. Vol. 2, Edited and with a foreword by Pierre
  Cartier and Catherine Chevalley, With a postface by J.-P. Bourguignon},
      review={\MR{1636473}},
}

\bib{CorletteFlatGBundles}{article}{
      author={Corlette, Kevin},
       title={Flat {$G$}-bundles with canonical metrics},
        date={1988},
     journal={J. Diff. Geom.},
      volume={28},
       pages={361\ndash 382},
}

\bib{Hamlet}{article}{
      author={Hamlet, Oskar},
       title={Tight holomorphic maps, a classification},
        date={2013},
        ISSN={0949-5932},
     journal={J. Lie Theory},
      volume={23},
      number={3},
       pages={639\ndash 654},
      review={\MR{3115169}},
}

\bib{Harish}{article}{
      author={Harish-Chandra},
       title={Representations of semisimple {L}ie groups. {VI}. {I}ntegrable
  and square-integrable representations},
        date={1956},
        ISSN={0002-9327},
     journal={Amer. J. Math.},
      volume={78},
       pages={564\ndash 628},
      review={\MR{0082056}},
}

\bib{Helgason}{book}{
      author={Helgason, Sigurdur},
       title={Differential geometry, {L}ie groups, and symmetric spaces},
      series={Graduate Studies in Mathematics},
   publisher={American Mathematical Society, Providence, RI},
        date={2001},
      volume={34},
        ISBN={0-8218-2848-7},
         url={http://dx.doi.org/10.1090/gsm/034},
        note={Corrected reprint of the 1978 original},
      review={\MR{1834454}},
}

\bib{humphreys}{book}{
      author={Humphreys, J.~E.},
       title={Linear algebraic groups},
      series={Graduate Texts in mathematics},
   publisher={Springer New York},
        date={1975},
      volume={21},
}

\bib{igusa}{article}{
      author={Igusa, Jun-ichi},
       title={A classification of spinors up to dimension twelve},
        date={1970},
        ISSN={0002-9327},
     journal={Amer. J. Math.},
      volume={92},
       pages={997\ndash 1028},
      review={\MR{0277558}},
}

\bib{KMrank2}{article}{
      author={Koziarz, V.},
      author={Maubon, J.},
       title={Representations of complex hyperbolic lattices into rank 2
  classical lie groups of hermitian type},
        date={2008},
     journal={Geom. Dedicata},
      volume={137},
       pages={85\ndash 111},
}

\bib{KMgeneraltype}{article}{
      author={Koziarz, V.},
      author={Maubon, J.},
       title={The {T}oledo invariant on smooth varieties of general type},
        date={2010},
        ISSN={0075-4102},
     journal={J. Reine Angew. Math.},
      volume={649},
       pages={207\ndash 230},
         url={http://dx.doi.org/10.1515/CRELLE.2010.093},
      review={\MR{2746471}},
}

\bib{KM}{article}{
      author={Koziarz, V.},
      author={Maubon, J.},
       title={Maximal representations of uniform complex hyperbolic lattices},
        date={2017},
     journal={Ann. of Math.},
      volume={185},
       pages={493\ndash 540},
        note={arXiv:1506.07274},
}

\bib{kostant}{article}{
      author={Kostant, Bertram},
       title={The cascade of orthogonal roots and the coadjoint structure of
  the nilradical of a {B}orel subgroup of a semisimple {L}ie group},
        date={2012},
        ISSN={1609-3321},
     journal={Mosc. Math. J.},
      volume={12},
      number={3},
       pages={605\ndash 620, 669},
      review={\MR{3024825}},
}

\bib{manivel}{article}{
      author={Manivel, L.},
       title={Configurations of lines and models of {L}ie algebras},
        date={2006},
        ISSN={0021-8693},
     journal={J. Algebra},
      volume={304},
      number={1},
       pages={457\ndash 486},
         url={http://dx.doi.org/10.1016/j.jalgebra.2006.04.029},
      review={\MR{2256401}},
}

\bib{mcgovern}{incollection}{
      author={McGovern, William~M.},
       title={The adjoint representation and the adjoint action},
        date={2002},
   booktitle={Algebraic quotients. {T}orus actions and cohomology. {T}he
  adjoint representation and the adjoint action},
      series={Encyclopaedia Math. Sci.},
      volume={131},
   publisher={Springer, Berlin},
       pages={159\ndash 238},
         url={http://dx.doi.org/10.1007/978-3-662-05071-2_3},
      review={\MR{1925831}},
}

\bib{Murakami}{article}{
      author={Murakami, Shingo},
       title={Sur certains espaces fibr\'es principaux diff\'erentiables et
  holomorphes},
        date={1959},
        ISSN={0027-7630},
     journal={Nagoya Math. J},
      volume={15},
       pages={171\ndash 199},
         url={http://projecteuclid.org/euclid.nmj/1118800246},
      review={\MR{0111063}},
}

\bib{nakagawa}{article}{
      author={Nakagawa, Hisao},
      author={Takagi, Ryoichi},
       title={On locally symmetric {K}aehler submanifolds in a complex
  projective space},
        date={1976},
        ISSN={0025-5645},
     journal={J. Math. Soc. Japan},
      volume={28},
      number={4},
       pages={638\ndash 667},
      review={\MR{0417463}},
}

\bib{Royden}{article}{
      author={Royden, H.~L.},
       title={The {A}hlfors-{S}chwarz lemma in several complex variables},
        date={1980},
        ISSN={0010-2571},
     journal={Comment. Math. Helv.},
      volume={55},
      number={4},
       pages={547\ndash 558},
         url={http://dx.doi.org/10.1007/BF02566705},
      review={\MR{604712}},
}

\bib{richardson}{article}{
      author={Richardson, Roger},
      author={R{\"o}hrle, Gerhard},
      author={Steinberg, Robert},
       title={Parabolic subgroups with abelian unipotent radical},
        date={1992},
        ISSN={0020-9910},
     journal={Invent. Math.},
      volume={110},
      number={3},
       pages={649\ndash 671},
         url={http://dx.doi.org/10.1007/BF01231348},
      review={\MR{1189494}},
}

\bib{Sampson}{article}{
      author={Sampson, J.~H.},
       title={Some properties and applications of harmonic mappings},
        date={1978},
     journal={Ann. Sci. {\'E}cole Norm. Sup.},
      volume={11},
       pages={211\ndash 228},
}

\bib{S1}{article}{
      author={Simpson, Carlos},
       title={Constructing variations of {H}odge structure using {Y}ang-{M}ills
  theory and applications to uniformization},
        date={1988},
     journal={J. Amer. Math. Soc.},
      volume={1},
       pages={867\ndash 918},
}

\bib{S2}{article}{
      author={Simpson, Carlos},
       title={Higgs bundles and local systems},
        date={1992},
     journal={Inst. Hautes \'Etudes Sci. Publ. Math.},
      volume={75},
       pages={5\ndash 95},
}

\bib{Siu}{article}{
      author={Siu, Y.-T.},
       title={The complex-analyticity of harmonic maps and the strong rigidity
  of compact {K}{\"a}hler manifolds},
        date={1980},
     journal={Ann. of Math.},
      volume={112},
       pages={73\ndash 111},
}

\bib{Wolf}{article}{
      author={Wolf, Joseph~A.},
       title={Fine structure of {H}ermitian symmetric spaces},
        date={1972},
       pages={271\ndash 357. Pure and App. Math., Vol. 8},
      review={\MR{0404716}},
}

\end{biblist}
\end{bibdiv}

\end{document}